\documentclass[11pt,a4paper]{amsart}
\usepackage[utf8]{inputenc}
\usepackage[english]{babel}
\usepackage{amsmath,amsfonts,amssymb,amsthm}
\usepackage{graphicx,color}
\usepackage{tikz}
\usepackage{subfigure}
\usepackage{enumerate}
\usepackage{upref} 
\usepackage{hyperref}
\usepackage[all]{hypcap} 

\hypersetup{%
  colorlinks=false,%
  citecolor=black,%
  filecolor=black,%
  linkcolor=black,%
  urlcolor=black,%
  pdfauthor={Luca Gugelmann, Reto Spoehel},%
  pdftitle={}%
}

\newtheorem{theorem}{Theorem}
\newtheorem{lemma}[theorem]{Lemma}

\newtheorem{remark}[theorem]{Remark}

\newtheorem{definition}[theorem]{Definition}

\newtheorem*{theorem*}{Theorem}
\newtheorem*{lemma*}{Lemma}
\newtheorem*{claim*}{Claim}
\newtheorem*{definition*}{Definition}

\renewcommand{\P}{\mathbb{P}}
\newcommand{\E}{\mathbb{E}}
\newcommand{\Var}{\mathrm{Var}}

\providecommand{\abs}[1]{\lvert#1\rvert}

\newcommand{\argmin}{\operatornamewithlimits{arg\,min}}

\newcommand{\lr}{\ensuremath{\lambda_{r,\theta}}}
\newcommand{\lrs}{\ensuremath{\lambda_{r,\theta^*}}}
\newcommand{\Lr}{\ensuremath{\Lambda_{r,\theta}}}
\newcommand{\mr}{\ensuremath{\mu_{r,\theta}}}
\newcommand{\mrs}{\ensuremath{\mu_{r,\theta^*}}}
\newcommand{\drs}{\ensuremath{d^{r*}}}
\newcommand{\pov}{\ensuremath{p_0(r,F,n)}}
\newcommand{\Fpi}{\ensuremath{F_r^\pi}}
\newcommand{\clk}{\ensuremath{C_{\ell,k}}}
\newcommand{\clkp}{\ensuremath{C_{\ell,k}^{r+}}}
\newcommand{\clks}{\ensuremath{C_{\ell,k}^{r*}}}

\newcommand{\Gnp}{\ensuremath{G(n,p)}}
\newcommand{\Grnp}{\ensuremath{G^r(n,p)}}
\newcommand{\cG}{\mathcal{G}}
\newcommand{\cS}{\mathcal{S}}
\newcommand{\RR}{\mathbb{R}}
\newcommand{\Jhat}{\hat J}
\newcommand{\cvH}{\vec{\mathcal{H}}}
\newcommand{\vH}{\vec{H}}
\newcommand{\ind}[1]{\mathbf{1}_{\{#1\}}}

\renewcommand{\phi}{\varphi}
\renewcommand{\epsilon}{\varepsilon}

\begin{document}

\title{On Balanced Coloring Games in Random Graphs}
\author{Luca Gugelmann}
\email{lgugelmann@inf.ethz.ch}
\address{Institute of Theoretical Computer Science\\ETH Zürich\\8092
  Zürich\\Switzerland}
\author{Reto Spöhel}
\email{reto.spoehel@bfh.ch}
\address{Technik und Informatik\\Berner Fachhochschule\\3400 Burgdorf, Switzerland}

\begin{abstract}
  Consider the balanced Ramsey game, in which a player has \(r\) colors and
  where in each step \(r\) random edges of an initially empty graph on \(n\)
  vertices are presented. The player has to immediately assign a different
  color to each edge and her goal is to avoid creating a monochromatic copy of
  some fixed graph \(F\) for as long as possible. The Achlioptas game is
  similar, but the player only loses when she creates a copy of \(F\) in one
  distinguished color. We show that there is an infinite family of non-forests
  \(F\) for which the balanced Ramsey game has a different threshold than the
  Achlioptas game, settling an open question by Krivelevich et al. We also
  consider the natural vertex analogues of both games and show that their
  thresholds coincide for all graphs \(F\), in contrast to our results for the
  edge case.
\end{abstract}

\maketitle

\section{Introduction}

\subsection{The balanced Ramsey game}
Consider the following probabilistic one-player game. Starting with the empty
graph on \(n\) vertices, in each step \(r\) new edges are sampled uniformly
at random from all non-edges and inserted into the graph. The player -- we call
her Painter -- has $r$
colors at her disposal and must color these $r$ edges immediately subject to
the restriction that each color is assigned to exactly one of the $r$ edges.
Her goal is to avoid creating a monochromatic copy of some fixed graph \(F\)
for as long as possible. We refer to this game as the \emph{balanced Ramsey
  game}; it was introduced by Marciniszyn et al.\ in~\cite{Marciniszyn2007a}.

The typical duration of this game when played with an optimal strategy is
formalized by the notion of its \emph{threshold function}
\(N_0(F,r,n)\). Specifically, we say that $N_0(F,r,n)$ is a threshold function
for the game (for a fixed graph $F$ and a fixed integer $r\geq 2$) if for any
function \(N(n)\ll N_0\)\footnote{We write $f\ll g$ for $f=o(g)$ and $f\gg g$ for $f=\omega(g)$. All our asymptotics are for $n\to\infty$.}, Painter can a.a.s.\footnote{asymptotically almost
  surely, i.e.\ with probability tending to 1 as $n$ tends to infinity}
`survive' for at least $N$ steps using an appropriate strategy, and if for
any \(N(n) \gg N_0 \), Painter a.a.s.\ cannot survive for more than $N$
steps regardless of her strategy.  Note that this defines the threshold
function only up to constant factors; therefore, whenever we compare two
threshold functions and e.g.\ say that one is strictly higher than the other
this refers to their orders of magnitude.

Standard arguments show that such a threshold function always exists for games
of this type (see \cite[Lemma~2.1]{Marciniszyn2009a}). Therefore the goal when studying these games usually is to
determine their threshold function \emph{explicitly}.
In~\cite{Marciniszyn2007a}, Marciniszyn et al.\ determined the threshold function of the
balanced Ramsey game for the case when \(F\) is a cycle of arbitrary fixed length, and \(r = 2\) colors are available. For example, the threshold of the balanced Ramsey game when $F=C_3$ is a triangle and $r=2$ was shown to be
$N_0(C_3,2,n)=n^{6/5}$.
More recently, Prakash et al.\
\cite{Prakash2009} extended these results to an arbitrary number of colors \(r
\geq 2\). In particular, their work yields the first threshold results for the case where $F=K_\ell$ is a complete graph of size at least $4$ (and $r$ is large enough; specifically, their result requires  \(r \geq \ell\)).

\subsection{The Achlioptas game}
\label{sec:achlioptas}
A similar game to the one above was first studied by Krivelevich et al.\ in
\cite{krivelevich2009avoiding}. For the purposes of this paper we shall refer
to it simply as the \emph{Achlioptas game}. In this game we start with an
empty graph on \(n\) vertices. In each step, \(r\) edges chosen uniformly at random from all edges never seen
before are revealed. The player has to choose exactly one of these edges for inclusion
in the graph; the remaining \(r-1\) edges are discarded. The player's
goal is to avoid creating a copy of some fixed graph \(F\) for as long as
possible. Note that this can be seen as a balanced Ramsey game with relaxed
rules such that the player only needs to worry about copies of \(F\) in the
first color and can ignore the other \(r-1\) colors. As an immediate
consequence, for any $F$ and $r$ the threshold of the Achlioptas game is an
upper bound on the threshold of the balanced Ramsey game.

Mütze et al.~\cite{Mutze2011} recently determined the general threshold function of the Achlioptas game, valid for any fixed graph \(F\) and any fixed integer \(r \geq 2\). The general threshold formula turns out to be considerably more complicated than the preliminary results of~\cite{krivelevich2009avoiding} suggest.

It follows from known results that if $F$ is e.g.\ a star or a
path, the balanced Ramsey and the Achlioptas game have different thresholds (see Section~\ref{sec:edge-case} for an
example). However, for all non-forests $F$ where both thresholds are known
(i.e.\ for all cases covered by Prakash et al.\ \cite{Prakash2009}), the two
thresholds coincide, and so far it was unknown whether in fact the
two thresholds coincide for any non-forest $F$ and any $r\geq 2$. This
question was raised explicitly in Krivelevich et al.\ \cite{Krivelevich2010}.
We answer this question negatively in this work.

\begin{theorem} \label{thm:1}
  There is an infinite family of non-forests \(F\) for which, for any fixed
  integer \(r\geq 2\), the balanced Ramsey game has a strictly lower threshold
  than the Achlioptas game.
\end{theorem}

The simplest non-forest graph \(F\) for which we show that the two online thresholds
differ consists of three triangles joined at a common vertex, cfr.\
Figure~\ref{fig:ex:f} on page~\pageref{fig:ex:f}.

Theorem~\ref{thm:1} is in contrast with known results on the \emph{offline} problems corresponding to the two online games discussed here: As shown in~\cite{Krivelevich2010}, the two offline
problems have the same threshold for `almost all' non-forests $F$, in particular for `most' graphs of the infinite family
from Theorem~\ref{thm:1}.\footnote{The
  result is proven for all non-forests \(F\) that have a strictly 2-balanced
  subgraph \(H \neq K_3\).}

\subsection{Vertex analogues}
\label{sec:vertex-analogues}

Both the balanced Ramsey and the Achlioptas game have a natural vertex
analogue, where the player is presented with \(r\) new \emph{vertices}
(instead of edges) in each step. At the start of these games, a random graph $G(n,p)$ on vertex set $\{v_1,\ldots, v_n\}$ is generated, hidden from the player's view, by including each of the $\binom{n}{2}$ possible edges with some fixed probability $p=p(n)$ independently. We assume that $r$ divides $n$. In each step of the game, the $r$ next consecutive vertices are revealed, along with all edges induced by the vertices revealed so far. Thus after $i$ steps, the player sees exactly the random edges induced by $v_1, \ldots, v_{ir}$.

In the balanced Ramsey game the player has to
assign each of $r$ available colors to exactly one of the \(r\) new vertices at each step, without completing a (vertex-)monochromatic copy of some fixed graph $F$. In the Achlioptas game, she has to select one of the $r$ new vertices, and the $r-1$ remaining vertices are discarded along with all incident edges. Again the player's goal is to avoid creating a copy of some fixed graph $F$. 

In both cases we are interested in explicit threshold functions \(p_0 = p_0(F,r,n)\) such
that (i) for any function \(p(n) \ll p_0\) there is a strategy which a.a.s.\ allows the player to color (resp.\ choose from) all \(n\)
vertices without creating a (monochromatic) copy of \(F\), and (ii) for any \(p(n) \gg p_0\) every possible player strategy a.a.s.\ fails to do so. (The mere existence of such threshold functions can again be shown similarly to \cite[Lemma~2.1]{Marciniszyn2009a}.)

Prakash et al.~\cite{Prakash2009} proved results analogous
to those discussed above for the edge-coloring setting also for the vertex case.  Moreover, also the results of
M\"utze et al.~\cite{Mutze2011} for the Achlioptas game translate with
minimal changes to the vertex setting, even though this is not made explicit
in their work. (We will elaborate on this in Section~\ref{sec:translation}
below and in the Appendix.) To sum up, in the literature the vertex and the edge case of the two games are equally well understood, and the known results for them are in complete analogy to each other.

As we shall see, this pattern breaks down in the general case: We prove that in the vertex case the thresholds of the balanced Ramsey and the Achlioptas game coincide for all graphs $F$ and all $r\geq 2$. This is in contrast with our result for the edge case given in Theorem~\ref{thm:1}.

\begin{theorem}[Main result]\label{thm:main-result}
  For all graphs $F$ and all $r\geq 2$, the vertex versions of the balanced Ramsey game and the Achlioptas game have the same threshold.
\end{theorem}

We give the explicit threshold formula of the two games in Section~\ref{sec:translation} below.

\subsection{Organization of this paper} Recall that the threshold of the (vertex) Achlioptas is always an upper bound on the threshold of the (vertex) balanced Ramsey game. Hence to prove Theorem~\ref{thm:main-result} it suffices to give an upper bound on the threshold of the vertex Achlioptas game and a matching lower bound on the vertex balanced Ramsey game.

In Section~\ref{sec:translation} we outline how the results of M\"utze et al.~\cite{Mutze2011} on the edge Achlioptas game, including their upper bound proof, translate to the vertex setting. The proofs for these results are given in the Appendix, as they follow their edge counterparts quite closely and are not the main contribution of this work. In Section~\ref{sec:tools}, we adapt some key concepts from~\cite{Mutze2011} to the vertex setting. In Section~\ref{sec:lower-bound}, we then use these to prove the desired matching lower bound for the vertex balanced Ramsey game. Finally, we prove Theorem~\ref{thm:1} concerning the edge case in Section~\ref{sec:edge-case}.

\section{On the vertex Achlioptas game}
\label{sec:translation}

In this section we adapt the formalism and the results of Mütze et al.\
\cite{Mutze2011} from the edge to the vertex case. The proofs are very similar; and we reproduce them in the Appendix. We also refer the reader to \cite{Mutze2011} for a more in-depth discussion of the intuition behind our threshold formulas.

A (vertex-)\emph{ordered graph} is a pair \((H,\pi)\), where \(H\) is a graph, \(h := v(H)\), and
\(\pi: V(H) \to \{1,\dotsc,h\}\) is an ordering of the vertices of $H$, conveniently denoted by its preimages, $\pi=(\pi^{-1}(1), \ldots, \pi^{-1}(h))$. In the context of the vertex Achlioptas or balanced Ramsey game, we interpret the ordering $\pi=:(u_1, \ldots, u_h)$ as the order in which the vertices
of \(H\) appeared in the process, where \(u_h\) is the
vertex that appeared first (the ``oldest'' vertex) and \(u_1\) is the vertex that appeared last (the ``youngest'' vertex).
We denote by \(\Pi(V(H))\) the set of all possible orderings of the vertices of
\(H\), and by
\begin{equation*}
  \mathcal{S}(F) := \bigl\{(H,\pi) \mid H\subseteq F \wedge \pi \in \Pi(V(H))\bigr\}
\end{equation*}
the set of all ordered subgraphs of \(F\).
For some ordered graph \((H,\pi)\) and a subgraph \(J\subseteq H\), we denote by
\(\pi|_J\) the order on the vertices of \(J\) induced by \(\pi\). Given an
ordered graph \((H,\pi)\), \(\pi=(u_1, \ldots, u_h)\), we denote by \(H\setminus \{u_1,\dotsc, u_i\}\) the graph obtained
from \(H\) by removing the vertices \(u_1, \dotsc, u_i\) and all edges that contain at least one one of these vertices. (In other words, \(H\setminus \{u_1,\dotsc, u_i\}\) is the subgraph of $H$ induced by the vertices $u_{i+1}, \ldots, u_h$.) We use $u\in H$ as a shorthand notation for $u\in V(H)$.

For any graph $H$, we use the notations $e(H):=|E(H)|$ and $v(H):=|V(H)|$. For any nonempty ordered graph \((H_1,\pi)\), \(\pi = (u_1, u_2, \dotsc, u_h)\), any
sequence of subgraphs \(H_2, \dotsc, H_h \subseteq H_1\) with \(H_i \subseteq
H_1 \setminus \{u_1, \dotsc, u_{i-1}\}\) and \(u_i \in H_i\) for all \(2 \leq
i \leq h\), and any integer \(r \geq 2\) define coefficients \(c_i = c_i(
(H_1,\pi), H_2, \dotsc, H_h, r)\) recursively by
\begin{equation}
  \label{eq:def-ci}
  \begin{split}
    c_1 &:= r,\\
    c_i &:= (r-1)\cdot\sum_{j=1}^{i-1}c_j\ind{u_i \in H_j}, \quad\quad
    2\leq i \leq h
  \end{split}
\end{equation}
(where \(\ind{u_i \in H_j}=1\) if \(u_i \in H_j\) and \(\ind{u_i \in H_j}=0\)
otherwise), and set
\begin{equation}
  \label{eq:def-drstar}
  \drs(H_1, \pi) :=
  \max_{\substack{H_2, \dotsc, H_h\\\forall i\geq 2: H_i \subseteq H_1\setminus
      \{u_1, \dotsc, u_{i-1}\}\wedge u_i \in H_i}}
    \frac{\sum_{i = 1}^{h}c_i e(H_i)}{1 + \sum_{i = 1}^{h}c_i \bigl(v(H_i)-1\bigr)}.
\end{equation}
Furthermore, we set for any integer \(r \geq 2\) and any nonempty graph \(F\),
\begin{equation}
  \label{eq:def-mrstar}
  m^{r*}(F) := \min_{\pi \in \Pi(V(F))}\max_{H_1\subseteq F} \drs(H_1, \pi|_{H_1}).
\end{equation}

With these notations and definitions, the main result of~\cite{Mutze2011} translates to the following statement for the vertex Achlioptas case:

\begin{theorem}
  \label{thm:achlioptas-ubound}
  Let \(F\) be a fixed nonempty graph, and let \(r \geq 2\) be a fixed
  integer. Then the threshold of the vertex Achlioptas game with parameters \(F\) and \(r\) is
  $$p_0(F,r,n)= n^{-1/m^{r*}(F)}.$$

  \noindent In particular, if \(p(n) \gg n^{-1/m^{r*}(F)}\), the player a.a.s.\ loses the
  vertex Achlioptas game with parameters \(F\) and \(r\), regardless of her
  strategy.
\end{theorem}

As discussed in the introduction, this result also yields an upper bound of $n^{-1/m^{r*}(F)}$ on the threshold of the vertex balanced Ramsey game with parameters $F$ and $r$. We will prove a matching lower bound in the next section. We now present an alternative formulation of Theorem~\ref{thm:achlioptas-ubound} that is more convenient for this lower bound proof. Again we refer to~\cite{Mutze2011} for a discussion of the advantages of this alternative viewpoint.

Given an ordered graph $(H,\pi)$, $\pi=(u_1, \ldots, u_h)$, we use $H\setminus u_1$ as a shorthand notation for $H\setminus \{u_1\}$, and $\pi\setminus u_1$ as a shorthand notation for $\pi|_{H\setminus \{u_1\}}$. As usual we denote for $u\in V(H)$ by $\deg_H(u)$ the degree of $u$ in~$H$.

For a fixed integer \(r \geq 2\) and a fixed real value \(0 \leq \theta \leq 2\) we
recursively define for any ordered graph \((H,\pi)\), \(\pi = (u_1, \dotsc,
u_h)\), the following quantity:
\begin{equation}
  \label{eq:lambda-r}
  \lr(H,\pi) := \left\{
  \begin{array}{l}
    0, \hfill\text{if \(v(H) = 0\)}\\
    \begin{split}
      1 &+ \Big(\lr(H\setminus u_1, \pi \setminus u_1)  - \theta \cdot\deg_{H}(u_1)\Big)
      \\&+ (r-1)\cdot\min_{\substack{J\subseteq H\\u_1 \in
          J}}\Big(\lr(J\setminus u_1, \pi|_{J\setminus u_1}) -
      \theta\cdot \deg_{J}(u_1)\Big),\hspace{15pt}
    \end{split}
    \\\hfill\text{otherwise.}
  \end{array}\right.
\end{equation}

We further define for \(r\) and \(\theta\) as before and any \(F\) the
quantity
\begin{equation}
  \label{eq:Lambda-r-theta}
  \Lambda_{r,\theta}(F) := \max_{\pi\in \Pi(V(F))}\min_{H\subseteq F}\lr(H,\pi|_H).
\end{equation}
It is straightforward to check that as a function of \(\theta\) for a fixed
\(r\) and a fixed nonempty graph \((H,\pi)\) respectively \(F\), both
\(\lr(H,\pi)\) and \(\Lr(F)\) are continuous, piecewise linear with integer coefficients, and non-increasing. Furthermore, both functions have a unique rational root.

Analogously to~\cite{Mutze2011} one can prove:

\begin{theorem}\label{thm:threshold-equivalence}
  Let \(F\) be a fixed nonempty graph, and let \(r\geq 2\) be a fixed
  integer. Let \(\theta^* = \theta^*(F,r)\) be the unique solution of
  \begin{equation}
    \Lr(F) \stackrel{!}{=} 0,
  \end{equation}
  where \(\Lr(F)\) is defined in \eqref{eq:lambda-r} and
  \eqref{eq:Lambda-r-theta}. Then we have
  \begin{equation*}
    m^{r*}(F) = \frac{1}{\theta^*(F,r)}\ .
  \end{equation*}
  Consequently, the threshold of the vertex Achlioptas game with parameters $F$ and $r$ can be written as
  $$p_0(F,r,n)= n^{-\theta^*(F,r)}.$$
\end{theorem}

\section{\texorpdfstring{$r$-matched graphs}{r-matched graphs}}
\label{sec:tools}

In this section we adapt some key notions concerning $r$-(edge-)matched graphs introduced in~\cite{Krivelevich2010} and~\cite{Mutze2011} to the vertex setting studied here. We will need these concepts in our proof of a lower bound on the vertex balanced Ramsey threshold.

\begin{definition}[\(r\)-matched graph]
  An \emph{\(r\)-(vertex-)matched graph} \(G = (V, E, \mathcal{K})\) is a (simple,
  undirected) graph with vertex set \(V\) and edge set \(E\) together with a
  partition \(\mathcal{K}\) of\/ \(V\) into sets of size \(r\), the
  \emph{\(r\)-sets}. With \(\kappa(G) := \abs{\mathcal{K}}=|V|/r\) we denote the
  number of \(r\)-sets of\/ \(G\).  We refer to the (non-\(r\)-matched) graph
  \(G' = (V, E)\) as the \emph{underlying graph of \(G\)}.
\end{definition}

We extend standard notions like graph isomorphism, subgraph containment etc.\ to \(r\)-matched graphs in the obvious way.

Recall that the vertex Achlioptas and vertex balanced Ramsey game is played on a binomial random graph $\Gnp$ on vertex set $\{v_1, \ldots, v_n\}$ that is initially hidden from the player's view and revealed $r$ vertices at a time. We denote by $G_i$ the graph induced by $\{v_1, \ldots, v_{ir}\}$ (i.e. the graph visible to the player after $i$ steps), viewed as an (uncolored) $r$-matched graph with partition $\mathcal{K}=\{\{v_1,\ldots, v_r\},\{v_{r+1}, \ldots, v_{2r}\}, \ldots, \{v_{(i-1)r+1},\ldots,v_{ir}\}\}$. In particular, $G_{n/r}$ is the random graph $\Gnp$ generated before the game starts, \emph{viewed as an $r$-matched graph} with the obvious partition. We denote a generic instance of such a random $r$-matched graph by $\Grnp$ in the following.

In our lower bound proof we will need the following simple lemma.

\begin{lemma}\label{lem:r-matched-expected}
  Let \(r\geq 2\) be a fixed integer, and let \(F\) be a fixed \(r\)-matched graph
  with at least one edge. Then the expected number of copies of \(F\) in \(G^r(n,p)\) is
    \(\Theta(n^{\kappa(F)}p^{e(F)})\).

\end{lemma}

\begin{proof}
There are $\binom{n/r}{\kappa(F)} \cdot \Theta(1) = \Theta (n^{\kappa(F)})$ possible occurrences of $F$ in $\Grnp$, and each of them appears with probability $p^{e(F)}$.
\end{proof}

For $r\geq 2$, any \(r\)-matched graph
\(F\) and \(0 \leq \theta \leq 2\) let
\begin{equation}
  \label{eq:def-mu-r-theta}
  \mu_{r,\theta}(F) := \kappa(F) - \theta \cdot e(F).
\end{equation}
Note that, by the above lemma, for \(p:= n^{-\theta}\) the expected number of copies of \(F\) in \(G^r(n,p)\) is of order $n^{\mu_{r,\theta}(F)}$.

\section{A matching lower bound on the vertex balanced Ramsey threshold}
\label{sec:lower-bound}

In this section we prove the main contribution of this work, a lower bound on the threshold of  the balanced Ramsey game that matches the upper bound given by Theorem~\ref{thm:achlioptas-ubound}. In view of Theorem~\ref{thm:threshold-equivalence}, it suffices to prove the following statement.

\begin{theorem} \label{thm:balanced-lbound} Let \(F\) be a fixed
  nonempty graph, and let \(r \geq 2\) be a fixed integer. Let \(\theta^* =
  \theta^*(F,r)\) be the unique solution of
  \begin{equation}
    \label{eq:theta-star-def}
    \Lr(F) \stackrel{!}{=} 0,
  \end{equation}
  where \(\Lr(F)\) is defined in \eqref{eq:lambda-r} and
  \eqref{eq:Lambda-r-theta}. Then for all \(p \ll n^{-\theta^*}\) there exists
  a strategy such that Painter can a.a.s.\ win the vertex balanced Ramsey game
  with parameters $F$ and $r$.
\end{theorem}

We now describe the general coloring strategy for which we will prove Theorem~\ref{thm:balanced-lbound}. The strategy is a natural extension of the one proposed in~\cite{Mutze2011} for the (edge or vertex) Achlioptas game; and we use very similar notations and conventions in the following. Note however that the analogous extension of the \emph{edge} Achlioptas strategy fails to yield a similar lower bound, cfr.\ Theorem~\ref{thm:1} and its proof in Section~\ref{sec:edge-case}.

Crucially, our strategy keeps track of the \emph{order} in which copies of subgraphs appear on the board. We say that the board contains a (monochromatic) copy of $(H,\pi)$, $\pi=(u_1, \ldots, u_h)$, if it contains a (monochromatic) subgraph isomorphic to $H$ whose vertices appeared in the order specified by $\pi$ (with $u_h$ being the first and $u_1$ being the last vertex to appear).

Let $r\geq 2$ and $0\leq \theta \leq 2$ be arbitrary but fixed. (Eventually we will set $\theta=\theta^*(F,r)$, but for the moment it is more convenient to work with an arbitrary $\theta$.) We denote with \(C\) the set of available colors.
Consider a fixed step of the game, and let \(R\) denote
the \(r\)-set presented to the player in that step. (We have $R=\{v_{(i-1)r+1}, \ldots, v_{ir}\}$ for some $i$, $1\leq i\leq n/r$.) Painter's decision in this step can be formalized
as choosing a perfect matching in the complete bipartite graph \(B\) with parts
\(C\) and \(R\), where each edge corresponds to assigning a color to a vertex. We say that a perfect matching \(M\) \emph{closes a copy} of some graph
\((H,\pi) \in \mathcal{S}(F)\) if coloring \(R\) according to \(M\) creates a monochromatic copy of \((H,\pi)\) on the board (clearly, then the last vertex of $H$ according to $\pi$ is in $R$).

Painter's strategy now is the following: She partitions $B$ into \(r\) disjoint perfect matchings \(M_1, \dotsc, M_r\) arbitrarily. (By an easy application of the marriage theorem, this is always possible.) For each of these matchings she determines the value
\begin{equation}
  \label{eq:weights}
    d(M) := \min\bigl\{\lr(H,\pi) \mid\ \text{\((H,\pi)\in \mathcal{S}(F)\)
      \(\wedge\) \(M\) closes a copy of
      \((H,\pi)\)}\bigr\},
\end{equation}
and chooses the matching for which this value is maximal.

If there is not a unique maximum, ties are broken according to the
following somewhat technical criterion.
Consider the directed graph $\cG=\cG(F)$ with vertex set $\cS(F)$ and arcs given by proper (ordered) subgraph inclusion; i.e., from every vertex $(H, \pi)$ there are arcs to all vertices $(J,\pi|_J)$ with $J\subsetneq H$. Clearly, $\cG$ contains no directed cycles. We extend $\cG$ to a graph $\cG'=\cG'(F,r,\theta)$ by first connecting every pair of distinct vertices $(H_1, \pi_1)$, $(H_2, \pi_2)$ for which $\lambda_{r,\theta}(H_1, \pi_1)=\lambda_{r,\theta}(H_2, \pi_2)$ with an (undirected) edge, and then orienting these additional edges in such a way that the directed graph $\cG'$ remains acyclic. (It is easy to see that this is always possible.) Note that for every fixed $\lambda_0\in\RR$ this yields a total ordering on all graphs $(H,\pi)$ with $\lambda_{r,\theta}(H,\pi)=\lambda_0$. We say that $(H_1, \pi_1)$ is higher than $(H_2, \pi_2)$ in this ordering if the corresponding arc in $\cG'$ is directed from $(H_1, \pi_1)$ to $(H_2, \pi_2)$.

Our strategy breaks ties according to this ordering: Whenever we have a choice between different perfect matchings with the same value $d(M)$, then for each such matching we consider the set of ordered graphs
\begin{equation} \label{def:mathJ-M}
    \mathcal{J}(M) := \argmin\{\lr(H,\pi) \mid (H,\pi) \in \mathcal{S}(F) \wedge \text{\(M\)
      closes a copy of \((H,\pi)\)}\}
\end{equation}
and, among these, we let \(J(M) \in \mathcal{J}(M)\) denote the graph that is \emph{lowest} in the total
ordering for $\lambda_0:=d(M)$. Then we select the matching $M$ for which $J(M)$ is \emph{highest} in the total ordering for $\lambda_0$.

The next lemma states a witness graph invariant that is crucial in our proof of
Theorem~\ref{thm:balanced-lbound}. Note that the statement of the lemma is purely deterministic.

\begin{lemma}
  \label{lem:lower-bound-main-lemma}
  Let \(r \geq 2\) be an integer and \(0 \leq \theta \leq 2\) be
  fixed. Following the above vertex coloring strategy ensures that the
  following invariant is maintained throughout the game for some \(v_{\max} =
  v_{\max}(F,r,\theta)\):

  The graph $G_i$ contains a copy of some \(r\)-matched graph \(K'\) with
  \(v(K')\leq v_{\max}\) and
  \begin{equation*}
    \mr(K') < 0,
  \end{equation*}
  or for every \((H,\pi) \in \mathcal{S}(F)\) we have that each monochromatic
  copy of \((H,\pi)\) on the board is contained in an \(r\)-matched subgraph
  \(H'\) of\/ $G_i$ with \(v(H') \leq v_{\max}\) and
  \begin{equation}
    \label{eq:invariant-inequality}
    \mr(H') \leq \lr(H,\pi),
  \end{equation}
  where \(\mr()\) and \(\lr()\) are defined in
  \eqref{eq:def-mu-r-theta} and \eqref{eq:lambda-r}, respectively.
\end{lemma}

We postpone the proof of Lemma~\ref{lem:lower-bound-main-lemma} and show first how it implies Theorem~\ref{thm:balanced-lbound}.

\begin{proof}[Proof of Theorem~\ref{thm:balanced-lbound}]
  Let \(\theta^* = \theta^*(F,r)\) be defined as in the theorem. We show that
  the above strategy for \(\theta := \theta^*\) allows Painter to win a.a.s.\
  for all \(p \ll \pov = n^{-\theta^*}\).

  By the definition of \(\theta^*\) (cf.\ \eqref{eq:Lambda-r-theta}
  and \eqref{eq:theta-star-def}) we have that for each possible ordering
  \(\pi\) of the vertices of \(F\) there exists some pair \((H,\pi|_H) \in
  \mathcal{S}(F)\) such that \(\lrs(H,\pi|_H) \leq 0\). According to
  Lemma~\ref{lem:lower-bound-main-lemma} the following holds for each such
  \((H,\pi|_H)\): If the final board contains a monochromatic copy of
  \((H,\pi|_H)\), then $G_{n/r}$ contains an \(r\)-matched graph \(K'\) of
  size at most \(v_{\max}\) and \(\mrs(K') < 0\), or an \(r\)-matched graph
  \(H'\), again of size at most \(v_{\max}\), satisfying
  \begin{equation*}
    \mrs(H') \leq \lrs(H,\pi|_{H})\leq 0\ .
  \end{equation*}
  This yields a family \(\mathcal{W} = \mathcal{W}(F,\pi,r)\) of \(r\)-matched
  graphs \(W'\) satisfying \(\mu(W') \leq 0\) and \(v(W') \leq v_{\max}\) such
  that, deterministically, $G_{n/r}$ contains a graph from \(\mathcal{W}\)
  if the final board contains a monochromatic copy of \((F,\pi)\) (and hence also a
  copy of \((H,\pi|_H)\)). It follows that $G_{n/r}$ contains a graph from
  \(\mathcal{W}^* = \mathcal{W}^*(F,r) := \cup_{\pi \in
    \Pi(E(F))}\mathcal{W}(F,\pi,r)\) if the final board contains a monochromatic
  copy of \(F\). Moreover, as no graph in \(\mathcal{W}^*\) has more than
  \(v_{\max}\) vertices, the size of \(\mathcal{W}^*\) is bounded by a
  constant depending only on \(F\) and \(r\). As $G_{n/r}$ is distributed as a random $r$-matched graph $\Grnp$, we obtain with
  Lemma~\ref{lem:r-matched-expected}, the definition of \(\mrs()\) in
  \eqref{eq:def-mu-r-theta}, and the fact that \(\mrs(W') \leq 0\) for all
  \(W'\in \mathcal{W}^*\), that for $p\ll n^{-\theta^*}$ the expected number of copies of graphs from
  \(\mathcal{W}^*\) in $G_{n/r}$ is of order
  \begin{equation*}
    \sum_{W'\in \mathcal{W}^*} n^{\kappa(W')}p^{e(W')}
    \ll \sum_{W'\in \mathcal{W}^*} n^{\mr(W')} \leq \abs{\mathcal{W}^*}\cdot n^0 = \Theta(1).
  \end{equation*}
  It follows from Markov's inequality that a.a.s.\ $G_{n/r}\cong \Grnp$ contains no
  $r$-matched graph from \(\mathcal{W}^*\). Consequently a.a.s.\ the final board contains no monochromatic copy of \(F\).
\end{proof}

For the proof of Lemma~\ref{lem:lower-bound-main-lemma} we require the
following technical lemma concerning the minimization in the definition of
\(\lr()\). The proof is straightforward and analogous to~\cite[Lemma 10]{Mutze2011}.

\begin{lemma}\label{lem:lr-drop-minimization}
  Let \(r \geq 2\) be an integer, \(0 \leq \theta \leq 2\) fixed, and let
  \(\mathcal{F}\) be a family of ordered graphs with the property that if some
  \((H,\pi)\), \(\pi = (u_1, \dotsc, u_h)\), is in \(\mathcal{F}\) then for
  every subgraph \(J \subseteq H\) with \(u_1 \in J\) also \((J, \pi|_J)\) is
  in \(\mathcal{F}\). Then for \(\lr()\) as defined in \eqref{eq:lambda-r} we
  have
  \begin{equation}
    \label{eq:argmin=argmin-u1}
    \argmin_{(H,\pi)\in \mathcal{F}} \lr(H,\pi) = \argmin_{(H,\pi)\in
      \mathcal{F}}\Bigl( \lr(H\setminus u_1, \pi\setminus u_1)- \theta\cdot \deg_H(u_1)\Bigr) \subseteq
    \mathcal{F},
  \end{equation}
  and all ordered graphs \((\hat J,\hat \pi)\), \(\hat\pi = (\hat u_1, \dotsc,
  \hat u_j)\), in the family \eqref{eq:argmin=argmin-u1} satisfy
  \begin{equation}\label{eq:min-nice-recursion}
    \lr(\hat J, \hat \pi) = 1 + r\cdot \bigl(\lr(\hat J\setminus \hat u_1, \hat \pi
    \setminus \hat u_1) - \theta\cdot\deg_{\hat J}(\hat u_1)\bigr).
  \end{equation}
  \qed
\end{lemma}


It remains to prove Lemma~\ref{lem:lower-bound-main-lemma}.

\begin{proof}[Proof of Lemma~\ref{lem:lower-bound-main-lemma}]
 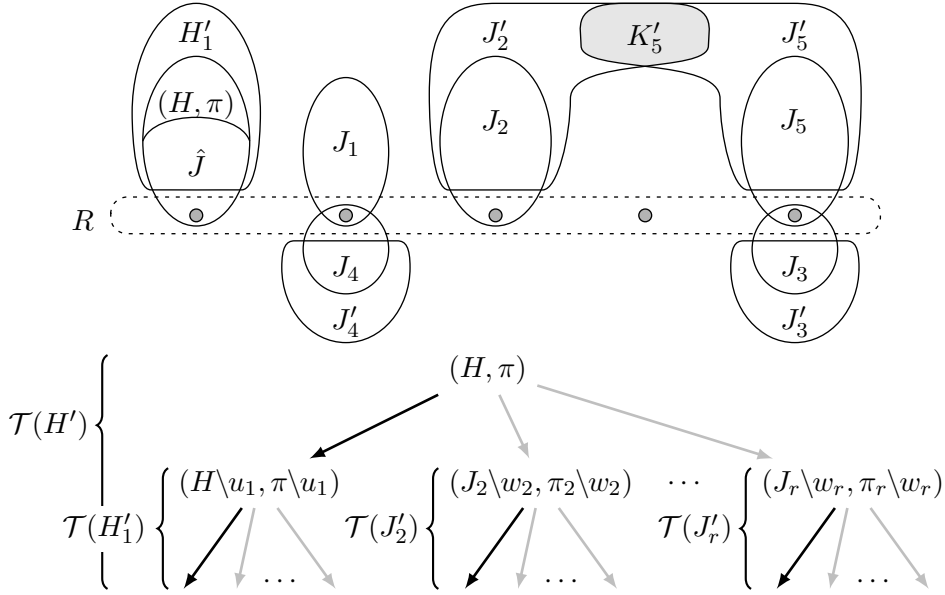
\begin{figure}
   \centering
\begin{tikzpicture}[y=0.80pt,x=0.80pt,yscale=-1, inner sep=0pt, outer sep=0pt]
  \definecolor{ce6e6e6}{RGB}{230,230,230}
  \definecolor{cb3b3b3}{RGB}{179,179,179}
  \path[fill=ce6e6e6] (370.4688,167.3622) .. controls (345.4688,167.3622) and
    (350.0000,179.8622) .. (350.0000,187.3622) .. controls (350.0000,193.6620) and
    (364.6748,194.2795) .. (380.0000,197.3622) .. controls (395.2196,194.2791) and
    (410.0000,193.6625) .. (410.0000,187.3622) .. controls (410.0000,179.8622) and
    (415.0000,167.3622) .. (390.0000,167.3622) -- cycle;
  \path[draw=black,dash pattern=on 1.00pt off 3.00pt,line join=bevel,line
    cap=rect,miter limit=4.00,line width=0.500pt,rounded corners=0.2167cm]
    (130.0000,258.7028) rectangle (490.0000,276.0012);
  \path[cm={{0.65997,0.0,0.0,0.65997,(27.33326,161.50509)}},draw=black,fill=cb3b3b3,line
    join=bevel,line cap=rect,miter limit=4.00,line width=0.500pt]
    (220.7183,160.3975)arc(-0.042:180.042:4.546)arc(-180.042:0.042:4.546) --
    cycle;
  \path[draw=black,line join=bevel,line cap=rect,miter limit=4.00,line
    width=0.500pt] (195.0000,232.3622) .. controls (195.0000,254.4536) and
    (183.8071,272.3622) .. (170.0000,272.3622) .. controls (156.1929,272.3622) and
    (145.0000,254.4536) .. (145.0000,232.3622) .. controls (145.0000,210.2708) and
    (156.1929,192.3622) .. (170.0000,192.3622) .. controls (183.8071,192.3622) and
    (195.0000,210.2708) .. (195.0000,232.3622) -- cycle;
  \path[draw=black,line join=bevel,line cap=butt,miter limit=4.00,line
    width=0.500pt] (145.0000,232.3622) .. controls (150.0000,217.3622) and
    (190.0000,217.3622) .. (195.0000,232.3622);
  \path[draw=black,line join=bevel,line cap=butt,miter limit=4.00,line
    width=0.500pt] (200.0000,218.8937) .. controls (200.0000,233.1238) and
    (198.0000,255.3622) .. (191.2132,255.3320) .. controls (185.0000,255.3043) and
    (178.2843,255.3622) .. (170.0000,255.3622) .. controls (161.7157,255.3622) and
    (155.0000,255.3044) .. (148.7868,255.3320) .. controls (142.0000,255.3622) and
    (140.0000,233.1238) .. (140.0000,218.8937) .. controls (140.0000,190.4336) and
    (153.4315,167.3622) .. (170.0000,167.3622) .. controls (186.5685,167.3622) and
    (200.0000,190.4336) .. (200.0000,218.8937) -- cycle;
  \path[fill=black] (170,242.36218) node[] (text4618) {$\hat{J}$};
  \path[fill=black] (170,215.36218) node[] (text4618-4) {$(H,\pi)$};
  \path[fill=black] (170,182.36218) node[] (text4618-4-4) {$H_1'$};
  \path[fill=black] (116.89404,269.76129) node[] (text4763) {$R$};
  \path[cm={{1.0,0.0,0.0,0.8,(275.0,50.47244)}},draw=black,line join=bevel,line
    cap=rect,miter limit=4.00,line width=0.500pt]
    (200.0000,227.3622)arc(0.000:180.000:25.000000 and
    50.000)arc(-180.000:0.000:25.000000 and 50.000) -- cycle;
  \path[fill=black] (449.54568,222.90787) node[] (text4618-4-9-0)
    {$J_5$};
  \path[fill=black] (450,182.36218) node[] (text4618-4-4-0-9) {$J_5'$};
  \path[cm={{1.0,0.0,0.0,0.7,(65.0,78.20867)}},draw=black,line join=bevel,line
    cap=rect,miter limit=4.00,line width=0.500pt]
    (195.0000,227.3622)arc(0.000:180.000:20.000000 and
    50.000)arc(-180.000:0.000:20.000000 and 50.000) -- cycle;
  \path[fill=black] (240,233.36218) node[] (text4618-4-9-5) {$J_1$};
  \path[cm={{0.65997,0.0,0.0,0.65997,(97.33326,161.50509)}},draw=black,fill=cb3b3b3,line
    join=bevel,line cap=rect,miter limit=4.00,line width=0.5pt]
    (220.7183,160.3975)arc(-0.042:180.042:4.546)arc(-180.042:0.042:4.546) --
    cycle;
  \path[cm={{1.0,0.0,0.0,0.8,(135.0,50.47244)}},draw=black,line join=bevel,line
    cap=rect,miter limit=4.00,line width=0.5pt]
    (201.0000,227.3622)arc(0.000:180.000:26.000000 and
    50.000)arc(-180.000:0.000:26.000000 and 50.000) -- cycle;
  \path[fill=black] (310,223.36218) node[] (text4618-4-9) {$J_2$};
  \path[fill=black] (310,182.36218) node[] (text4618-4-4-0) {$J_2'$};
  \path[cm={{1.0,0.0,0.0,0.42,(65.0,187.87006)}},draw=black,line join=bevel,line
    cap=rect,miter limit=4.00,line width=0.5pt]
    (195.0000,227.3622)arc(0.000:180.000:20.000000 and
    50.000)arc(-180.000:0.000:20.000000 and 50.000) -- cycle;
  \path[draw=black,line join=bevel,line cap=butt,miter limit=4.00,line
    width=0.500pt] (270.0000,292.0190) .. controls (270.0000,311.5385) and
    (256.5685,327.3622) .. (240.0000,327.3622) .. controls (223.4315,327.3622) and
    (210.0000,311.5385) .. (210.0000,292.0190) .. controls (210.0000,282.2593) and
    (212.0000,279.3622) .. (216.0000,279.3622) .. controls (220.0000,279.3622) and
    (231.7157,279.3622) .. (240.0000,279.3622) .. controls (248.2843,279.3622) and
    (260.0000,279.3622) .. (264.0000,279.3622) .. controls (268.0000,279.3622) and
    (270.0000,282.2593) .. (270.0000,292.0190) -- cycle;
  \path[fill=black] (240,292.36218) node[] (text4618-4-9-6) {$J_4$};
  \path[fill=black] (240,317.36218) node[] (text4618-4-4-0-6) {$J_4'$};
  \path[draw=black,line join=bevel,line cap=butt,miter limit=4.00,line
    width=0.500pt] (279.0000,210.3622) .. controls (279.0000,228.3114) and
    (281.0000,255.3622) .. (288.0000,255.3622) .. controls (293.0000,255.3622) and
    (300.4605,255.3622) .. (310.0000,255.3622) .. controls (319.5395,255.3622) and
    (327.3452,255.3622) .. (332.0000,255.3622) .. controls (339.0000,255.3622) and
    (345.0000,230.3114) .. (345.0000,212.3622) .. controls (345.0000,192.3622) and
    (409.5457,200.6829) .. (409.5457,187.3622) .. controls (409.5457,179.8622) and
    (414.5457,167.3622) .. (389.5457,167.3622) .. controls (365.6465,167.3622) and
    (314.7698,167.3622) .. (310.0000,167.3622) .. controls (290.9209,167.3622) and
    (279.0000,174.4637) .. (279.0000,210.3622) -- cycle;
  \path[fill=black] (380,183.36218) node[] (text5009) {$K_5'$};
  \path[cm={{0.65997,0.0,0.0,0.65997,(167.33326,161.50509)}},draw=black,fill=cb3b3b3,line
    join=bevel,line cap=rect,miter limit=4.00,line width=0.5pt]
    (220.7183,160.3975)arc(-0.042:180.042:4.546)arc(-180.042:0.042:4.546) --
    cycle;
  \path[cm={{0.65997,0.0,0.0,0.65997,(307.33326,161.50509)}},draw=black,fill=cb3b3b3,line
    join=bevel,line cap=rect,miter limit=4.00,line width=0.5pt]
    (220.7183,160.3975)arc(-0.042:180.042:4.546)arc(-180.042:0.042:4.546) --
    cycle;
  \path[cm={{0.65997,0.0,0.0,0.65997,(237.33326,161.50509)}},draw=black,fill=cb3b3b3,line
    join=bevel,line cap=rect,miter limit=4.00,line width=0.5pt]
    (220.7183,160.3975)arc(-0.042:180.042:4.546)arc(-180.042:0.042:4.546) --
    cycle;
  \path[cm={{1.0,0.0,0.0,0.42,(275.0,187.87006)}},draw=black,line join=bevel,line
    cap=rect,miter limit=4.00,line width=0.5pt]
    (195.0000,227.3622)arc(0.000:180.000:20.000000 and
    50.000)arc(-180.000:0.000:20.000000 and 50.000) -- cycle;
  \path[fill=black] (450,292.36218) node[] (text4618-4-9-6-1) {$J_3$};
  \path[fill=black] (450,317.36218) node[] (text4618-4-4-0-6-3)
    {$J_3'$};
  \path[draw=black,line join=bevel,line cap=butt,miter limit=4.00,line
    width=0.500pt] (481.0000,210.3622) .. controls (481.0000,228.3114) and
    (479.0000,255.3622) .. (472.0000,255.3622) .. controls (467.0000,255.3622) and
    (459.5395,255.3622) .. (450.0000,255.3622) .. controls (440.4605,255.3622) and
    (432.6548,255.3622) .. (428.0000,255.3622) .. controls (421.0000,255.3622) and
    (415.0000,230.3114) .. (415.0000,212.3622) .. controls (415.0000,192.3622) and
    (350.0000,200.6829) .. (350.0000,187.3622) .. controls (350.0000,179.3622) and
    (345.0000,167.3622) .. (370.0000,167.3622) .. controls (393.8992,167.3622) and
    (444.7759,167.3622) .. (449.5457,167.3622) .. controls (468.6247,167.3622) and
    (481.0000,174.4637) .. (481.0000,210.3622) -- cycle;
  \path[draw=black,line join=bevel,line cap=butt,miter limit=4.00,line
    width=0.500pt] (480.0000,292.0190) .. controls (480.0000,311.5385) and
    (466.5685,327.3622) .. (450.0000,327.3622) .. controls (433.4315,327.3622) and
    (420.0000,311.5385) .. (420.0000,292.0190) .. controls (420.0000,282.2593) and
    (422.0000,279.3622) .. (426.0000,279.3622) .. controls (430.0000,279.3622) and
    (441.7157,279.3622) .. (450.0000,279.3622) .. controls (458.2843,279.3622) and
    (470.0000,279.3622) .. (474.0000,279.3622) .. controls (478.0000,279.3622) and
    (480.0000,282.2593) .. (480.0000,292.0190) -- cycle;
\end{tikzpicture}\\
\usetikzlibrary{arrows,shapes,decorations.pathreplacing}
\begin{tikzpicture}[a/.style={-latex,line width=1pt},
	b/.style={thick,decorate,decoration={brace, amplitude=4pt}}]
	\newcommand{\sm}{\hspace{-3pt}\setminus\hspace{-3pt}}
	\node[] (hpi) at (3,1.5) {$(H,\pi)$};
	\node[] (hm) at (0,0) {$(H\sm u_1,\pi\sm u_1)$};
	\node[] (j2) at (3.7,0) {$(J_2\sm w_2,\pi_2\sm w_2)$};
	\node[] (jr) at (7.8,0) {$(J_r\sm w_r,\pi_r\sm w_r)$};
	\path (j2) -- (jr) node[pos=0.4] {$\cdots$};
	\draw[a] (hpi) -- (hm);
	\draw[a,gray!50] (hpi) -- (j2);
	\draw[a,gray!50] (hpi) -- (jr);
	\foreach \x in {hm,j2,jr} {
		\draw[-latex,line width=1pt] (\x) -- +(-1,-1.4);
		\draw[-latex,line width=1pt,gray!50] (\x) -- +(-0.3,-1.4);
		\draw[-latex,line width=1pt,gray!50] (\x) -- +(1,-1.4);
		\draw (\x) +(0.3,-1.3) node {$\cdots$};
	}
	\draw[b,decoration={aspect=0.7}] (hpi) +(-5,-2.9) -- +(-5,0.2)  node[pos=0.7,left=3pt] {$\mathcal{T}(H')$};
	\draw[b] (hm.west) +(0,-1.4) -- +(0,0.2)  node[pos=0.5,left=5pt,fill=white,inner sep=3pt] {$\mathcal{T}(H_1')$};
	\draw[b] (j2.west) +(0,-1.4) -- +(0,0.2)  node[pos=0.5,left=3pt] {$\mathcal{T}(J_2')$};
	\draw[b] (jr.west) +(0,-1.4) -- +(0,0.2)  node[pos=0.5,left=3pt] {$\mathcal{T}(J_r')$};
\end{tikzpicture}
   \caption{Notations used in the proof of
     Lemma~\ref{lem:lower-bound-main-lemma}. The arcs of \(\mathcal{T}(H')\)
     drawn grey are either grey or red in the proof.}
   \label{fig:main-lemma-figure}
 \end{figure}
  To simplify the notation we drop the subscripts from \(\lr\) and
  \(\mr\) and write \(\lambda\) and \(\mu\) instead.
  \renewcommand{\lr}{\lambda}\renewcommand{\mr}{\mu}%
  For the reader's convenience, Figure~\ref{fig:main-lemma-figure} illustrates
  the notations used throughout the proof.
  Let
  \begin{equation}
    \label{eq:epsilon-lambda}
    \begin{split}
      \epsilon = \epsilon(F,r,\theta) = \min\Bigl\{\abs{\lr(H_1,\pi_1) - \lr(H_2, \pi_2)}
      \;\mbox{\Large\(\mid\)}\; &(H_1,\pi_1), (H_2,\pi_2)\in \mathcal{S}(F) \\
      &\wedge \lr(H_1,\pi_1) \neq \lr(H_2, \pi_2) \Bigr\}
    \end{split}
  \end{equation}
  and
  \begin{equation}
    \label{eq:def-vmax}
    v_{\max} = v_{\max}(F,r,\theta) = r^{(v(F)r/\epsilon +
      1)\abs{\mathcal{S}(F)} + 2}\cdot v(F) + r
  \end{equation}

  We prove the lemma by induction on the number of steps in the game. We show
  that the statement about graphs \((H,\pi)\in \mathcal{S}(F)\) is true as
  long as the currently revealed graph $G_i$ does not contain an \(r\)-matched subgraph \(K'\) with
  \(v(K') \leq v_{\max}\) and \(\mr(K') < 0\). Once such a subgraph \(K'\)
  appears we are done as it will remain in the game to the end.

  After the first step Painter has assigned a color to \(r\) vertices and the
  inequality~(\ref{eq:invariant-inequality}) is trivially satisfied: In each color we only have a single vertex, which has a $\lambda$-value of $1$ according to~\eqref{eq:lambda-r}. Each of these vertices is contained in the $r$-matched graph induced by the first $r$-set, whose $\mu$-value is at most $1$, see~\eqref{eq:def-mu-r-theta}.

  Consider now an arbitrary step of the game, and denote with \(M_1, \dotsc,
  M_r\) the matchings Painter considered in this step, where w.l.o.g.\ \(M_1\) is the matching Painter chose. Assume that $M_1$ completed a monochromatic copy of \((H\setminus u_1,\pi\setminus u_1)\) to a copy of $(H,\pi)$ (where $u_1$ denotes the first vertex of $\pi$).  Let $\hat J$ be some graph in
$\argmin_{J\subseteq H, u_1 \in J} \lr(J, \pi|_{J})$, and note that $M_1$ also closed a copy of $\hat J$.
  For $1\leq i\leq r$, let $(J_i,\pi_i):=J(M_i)$ as in the definition of our strategy after \eqref{def:mathJ-M}. By definition \((J_1, \pi_1)\) minimizes \(\lr()\)
  over all monochromatic ordered graphs in \(\mathcal{S}(F)\) that are closed by \(M_1\), see \eqref{eq:weights}. Furthermore, since Painter preferred $M_1$ over the alternatives we have $\lr( J_1,\pi_1 ) \geq \lr(J_i,\pi_i)$, $2\leq i\leq r$.
  Taken together it follows that
  \begin{equation}
    \label{eq:lambda-J_1-geq-J_i}
   \lr(\hat J, \pi|_{\hat J}) \geq \lr( J_1,\pi_1 ) \geq \lr(J_i,\pi_i)\quad\text{for \(2\leq i\leq r\)}.
  \end{equation}
  Note that \(H\), \(\hat J\) or \(J_1\) might be the same graph.

  For $1\leq i\leq r$, let $w_i$ denote the youngest vertex of $J_i$ according to $\pi_i$; i.e., $\pi_i=(w_i, \ldots)$.
  Again by the definition of our strategy the graphs \((J_i, \pi_i)\) minimize \(\lr()\) among all graphs that are closed by $M_i$, \(1 \leq i \leq
  r\). As for each index $i$ the family of these graphs is subgraph-closed in the sense of
  Lemma~\ref{lem:lr-drop-minimization}, it follows that
  \begin{equation*}
    \lr(J_i,\pi_i) = 1 + r\Bigl(\lr(J_i\setminus w_i, \pi_i \setminus w_i) -
    \theta\cdot \deg_{J_i}(w_i)\Bigr).
  \end{equation*}
  Similarly, Lemma~\ref{lem:lr-drop-minimization} also yields that
  \begin{equation*}
    \lr(\Jhat,\pi|_{\Jhat}) = 1 + r\Bigl(\lr(\Jhat\setminus u_1, \pi|_{\Jhat\setminus u_1}) -
    \theta\cdot\deg_{\Jhat}(u_1)\Bigr).
 \end{equation*}

Applying these transformations to equation~(\ref{eq:lambda-J_1-geq-J_i}) yields that for \(1 \leq i \leq
  r\)
  \begin{equation}
    \label{eq:h1-circ-geq-hi-circ}
    \lr(\hat J \setminus u_1, \pi|_{\hat J\setminus u_1}) - \theta\cdot\deg_{\hat J}(u_1) \geq
    \lr(J_i\setminus w_i, \pi_i \setminus w_i) - \theta\cdot \deg_{J_i}(w_i).
  \end{equation}

  The copy of \((H \setminus u_1,\pi\setminus u_1)\) on the board is monochromatic and by
  induction must be contained in some \(r\)-matched graph \(H_1'\) satisfying
  equation~\eqref{eq:invariant-inequality}, i.e.
  \begin{equation}
    \label{eq:induction-mu-lambda-h1prime}
    \mr(H_1') \leq \lr(H\setminus u_1,\pi\setminus u_1)\ .
  \end{equation}
 Similarly, the copies of \((J_i \setminus w_i, \pi_i \setminus
  w_i)\) that are completed to copies of $(J_i,\pi_i)$, \(2 \leq i \leq r\) on the board are also monochromatic, and hence they are contained in \(r\)-matched graphs \(J_2',\dotsc,
  J_r'\) with
  \begin{equation}
    \label{eq:induction-mu-lambda}
    \mr(J_i') \leq \lr(J_i\setminus w_i, \pi_i \setminus w_i) \quad \text{for \(2 \leq i \leq r\)}.
  \end{equation}
  By induction all these graphs contain at most \(v_{\max}\) vertices. We can
  also assume that \(\mr(H_1')\) and \(\mr(J_2'), \ldots, \mr(J_r')\) are all non-negative, as
  otherwise we have found a graph \(K'\) with \(\mr(K') < 0\) and \(v(K') \leq
  v_{\max}\) and are done. We will argue later that if the $\mu$-values under consideration are indeed non-negative, even stronger bounds on the number of vertices hold; specifically,
  that
  \begin{equation}
    \label{eq:j_i-H-vmax-bound}
    \begin{split}
      v(H'_1) &< v_{\max}/r - 1\\
      v(J_i') &< v_{\max}/r - 1 \quad\text{for \(2\leq i \leq r\)}.
    \end{split}
  \end{equation}

  We now construct an \(r\)-matched graph \(H'\) for \((H,\pi)\) satisfying
  the conditions of the lemma. Denote with \(E_1\) a set of edges that completes the considered copy of $(H\setminus u_1,\pi\setminus u_1)$ to a copy of $(H,\pi)$ (where the vertex corresponding to $u_1$ is in $R$ and $|E_1|=\deg_H(u_1)$). Similarly, for \(2 \leq i \leq r\) denote with \(E_i\) a set of edges that completes the considered copy of $(J_i\setminus w_i,\pi_i\setminus w_i)$ to a copy of $(J_i,\pi_i)$ (where the vertex corresponding to $w_i$ is in $R$ and $|E_i|=\deg_{J_i}(w_i)$).

 Let \(H'\) be the \(r\)-matched graph obtained by the union of
  \(H_1'\), \(J_i'\) and \(R\) together with all the edges in \(E_i\), \(1\leq
  i \leq r\). Formally, we set
  \begin{equation*}
    \begin{split}
      V(H') &:= R \cup V(H_1') \cup \bigcup_{i=2}^r V(J_i')\\
      E(H') &:= E(H_1') \cup \bigcup_{i=2}^r E(J_i') \cup \bigcup_{i=1}^r
      E_i\\
      \mathcal{K}(H') &:= \{R\} \cup \mathcal{K}(H_1') \cup \bigcup_{i=2}^r
      \mathcal{K}(J_i')\ .
    \end{split}
  \end{equation*}
  This is again a well-defined \(r\)-matched graph: All \(r\)-sets in
  \(H_1'\), \(J_i'\) ($2\leq i\leq r$) and \(\{R\}\) are also \(r\)-sets of the current game
  board. As such they are either equal or disjoint. Further $V(H')$ is indeed the union of all $r$-sets in $\mathcal{K}(H')$, and contains the endpoints of all edges in $E(H')$.
	
  Note that by
  \eqref{eq:j_i-H-vmax-bound} it follows that \(v(H') < v_{\max}\).

  The $r$-matched graphs $H_1', J_2', \ldots, J_r'$ are all formed by $r$-sets that appeared before $R$ in the process and are therefore vertex-disjoint from $R$. In particular, they do not contain any edges from $E_1, \ldots, E_r$. Furthermore, the sets $E_1,\ldots, E_r$ are pairwise disjoint: if two such sets $E_{i_1}$, $E_{i_2}$ involve the same vertex from $R$, then together with this vertex they complete monochromatic copies of graphs $(J_i\setminus w_i,\pi_i\setminus w_i)$ in two \emph{different} colors to copies of $(J_i, \pi_i)$; i.e., the endpoints of the edges in $E_{i_1}$, $E_{i_2}$ outside $R$ are in two different colors and are therefore distinct.

We define the $r$-matched graphs
  \begin{equation*}
    K_i' = J_i' \cap \Bigl(H_1' \cup \bigcup_{j=2}^{i-1}J_j'\Bigr) \quad\text{for
      \(2\leq i \leq r\)}.
  \end{equation*}
With the above observations and the definition of \(\mr()\) in \eqref{eq:def-mu-r-theta} we obtain that
  \begin{equation}
    \label{eq:mu-H'-1}
    \begin{split}
      \mr(H') &= 1 + \mr(H_1') - \theta\cdot\deg_{H}(u_1) \\&\hphantom{=}\,+
      \sum_{i = 2}^{r}\bigl(\mr(J_i') - \theta\cdot\deg_{J_i}(w_i)\bigr) - \sum_{i = 2}^{r} \mr(K_{i}').
    \end{split}
  \end{equation}
  We can assume that all \(\mr(K_{i}')\) are non-negative, because if this is
  not the case we have found a graph \(K'\) with \(\mr(K') < 0\) and \(v(K')
  \leq v_{\max{}}\) and are done.
  With this observation and \eqref{eq:induction-mu-lambda-h1prime}, (\ref{eq:induction-mu-lambda}) we
  obtain that
  \begin{equation*}
    \begin{split}
      \mr(H') &\leq 1 + \lr(H\setminus u_1, \pi \setminus u_1) -
      \theta\cdot\deg_{H}(u_1) \\&\hphantom{\leq}\,+ \sum_{i =
        2}^{r}\bigl(\lr(J_i\setminus w_i, \pi_i \setminus w_i)
      - \theta\cdot\deg_{J_i}(w_i)\bigr).
    \end{split}
  \end{equation*}
  Combining this with equation~\eqref{eq:h1-circ-geq-hi-circ} yields
  \begin{equation}
    \label{eq:mu-H'-3}
    \begin{split}
      \mr(H') &\leq 1 + \lr(H\setminus u_1, \pi \setminus u_1) -
      \theta\cdot\deg_{H}(u_1) \\&\hphantom{\leq}\,+ (r-1)\cdot\Bigl(\lr(\hat J\setminus u_1,
      \pi|_{\hat J\setminus u_1}) - \theta\cdot\deg_{\hat J}(u_1)\Bigr).
    \end{split}
  \end{equation}
  By Lemma~\ref{lem:lr-drop-minimization} and our choice of \(\hat J\) the right hand side of the above equation equals
  \(\lr(H,\pi)\) as defined in~\eqref{eq:lambda-r}; i.e., we have
  \begin{equation*}
     \mr(H') \leq \lr(H,\pi)
  \end{equation*}
  as desired.

  It remains to prove that equation~(\ref{eq:j_i-H-vmax-bound}) holds. It
  suffices to show that given \(\mu(H') \geq 0\) we have \(v(H') \leq
  v_{\max}/r-1\).

  In the above argument we constructed \(H'\) from copies of \(H_1'\) and
  \(J_i'\), or in other words from graphs constructed equivalently to \(H'\)
  in prior steps of the induction from \((H\setminus u_1, \pi\setminus u_1)\)
  and \((J_i\setminus w_i, \pi_i \setminus w_i)\). To analyze this
  construction we associate it with an edge-colored directed rooted tree
  \(\mathcal{T}(H')\) (cf.\ Figure~\ref{fig:main-lemma-figure}). The vertices of \(\mathcal{T}(H')\) correspond to
  monochromatic copies of graphs from \(\mathcal{S}(F)\) on the board of the game (the same copy may appear as a vertex multiple times). If \((H,\pi)\) consists of a single
  vertex, then \(\mathcal{T}(H')\) consists just of the copy of \((H,\pi)\) as
  the root. If this is not the case, then \(\mathcal{T}(H')\) consists of
  the copy of \((H, \pi)\) as the root vertex joined to \(r\) subtrees
  \(\mathcal{T}(H_1')\) and \(\mathcal{T}(J_i')\), \(2 \leq i \leq
  r\). The subtree \(\mathcal{T}(H_1')\) is connected to the root by a black arc and every
  \(\mathcal{T}(J_i')\) is connected to the root by either a grey or red arc
  according to the following criterion: Each such arc corresponds to an
  instance of the inequalities in (\ref{eq:lambda-J_1-geq-J_i}) somewhere
  along the induction. The arc is grey if both inequalities are tight, i.e.,
  if \(\lr(\hat J,\pi|_{\hat J}) = \lr(J_i, \pi_i)\). If on the other hand at
  least one of the inequalities is strict, i.e., if \(\lr(\hat J, \pi|_{\hat
    J}) > \lr(J_i, \pi_i)\), then the arc is red. All arcs are oriented away
  from the root. Note that \(\mathcal{T}(H')\) captures only the logical
  structure of the inductive history of \(H'\). Overlappings (captured by the
  graphs \(K_{i}'\) in \eqref{eq:mu-H'-1}) are completely ignored.

  Every red arc of \(\mathcal{T}(H')\) corresponds to a strict inequality in
  \eqref{eq:lambda-J_1-geq-J_i}. In this case, as a consequence of
  Lemma~\ref{lem:lr-drop-minimization},
  equation~\eqref{eq:h1-circ-geq-hi-circ} is also strict, with a difference of
  at least \(\epsilon/r\) (cf.~\eqref{eq:epsilon-lambda}) between the right
  and left side. Consequently, each red arc contributes a term of
  \(-\epsilon/r\) to the right side of \eqref{eq:mu-H'-3} in the corresponding
  induction step. Accumulating these terms along the induction yields that
  \begin{equation}
    \label{eq:mr=lr-eps}
    \mr(H') \leq \lr(H,\pi) - \ell(H')\cdot\epsilon/r,
  \end{equation}
  where \(\ell(H')\) denotes the number of red arcs in \(\mathcal{T}(H')\).

  Note that \(\lr(H,\pi) \leq v(F)\) for all \((H,\pi) \in
  \mathcal{S}(F)\). Thus if \(\mu(H') \geq 0\), then by \eqref{eq:mr=lr-eps}
  the tree \(\mathcal{T}(H')\) has at most \(\lr(H,\pi)r/\epsilon \leq
  v(F)r/\epsilon\) many red arcs. We will show that, due to our tie-breaking
  rule involving the auxiliary graph \(\mathcal{G}'\), this bound on the
  number of red arcs implies the claimed bound of \(v_{\max}/r-1\) on the
  number of vertices of \(H'\). To that end, we first show that if two
  vertices of \(\mathcal{T}(H')\) are connected by a (directed, i.e.\
  descending) path \(P\) that contains no red arcs, then these two vertices
  correspond to copies of \emph{different} ordered graphs \((H_1,\pi_1)\),
  \((H_2,\pi_2) \in \mathcal{S}(F)\).

  Consider such a walk between two vertices \((H_1,\pi_1)\) and
  \((H_2,\pi_2)\). We can map \(P\) to a directed walk \(P'\) in
  \(\mathcal{G}'\) as follows. The initial vertex of \(P'\) is
  \((H_1,\pi_1)\). For each black arc in \(P\) from a copy of some \((H,\pi)
  \in \mathcal{S}(F)\) to a copy of \((H\setminus u_1, \pi\setminus u_1)\) we
  extend \(P'\) by an arc from \((H,\pi)\) to \((H\setminus u_1, \pi\setminus
  u_1)\). This arc exists in \(\mathcal{G}'\) by subgraph containment. For each
  grey arc in \(P\) from a copy of some graph \((H,\pi)\) to a copy of some graph \((J_i\setminus
  w_i, \pi_i \setminus w_i)\) for some \(2 \leq i\leq r\), we have
  \begin{equation}
    \label{eq:walk-lr-equalities}
    \lr(H,\pi) \geq \lr(\hat J, \pi|_{\hat J}) = \lr( J_1,\pi_1 ) = \lr(J_i,\pi_i).
  \end{equation}
  In \(\mathcal{G}'\) we can then walk between the first two graphs in the
  above equation (assuming that they are different) because the second is
  contained in the first. Further we can walk from the second to the third
  because \((J_1,\pi_1) = J(M_1)\), and therefore by definition it must be lower in the ordering
  than \((\hat J, \pi|_{\hat J})\), see the text just after \eqref{def:mathJ-M}. The walk between
  the last two graphs in (\ref{eq:walk-lr-equalities}) is possible because
  Painter chose the matching \(M_1\), and by our tie-breaking criterion this
  means that \((J_1,\pi_1) = J(M_1)\) is higher in the ordering than
  \((J_i,\pi_i) = J(M_i)\). The last arc between \((J_i,\pi_i)\) and
  \((J_i\setminus w_i,\pi_i\setminus w_i)\) is in \(\mathcal{G}\) by subgraph
  containment. We extend \(P'\) by all these arcs as well (if any two
  subsequent graphs in this walk are the same, then the corresponding step in
  the walk is skipped).  Proceeding in this manner we obtain a
  directed walk \(P'\) in \(\mathcal{G}'\) from \((H_1,\pi_1)\) to
  \((H_2,\pi_2)\). As \(\mathcal{G}'\) is acyclic we must have \((H_1,\pi_1)
  \neq (H_2, \pi_2)\).

  It follows that a (directed) path in \(\mathcal{T}(H')\) that contains no
  red arcs has at most \(\abs{\mathcal{S}(F)}\) many vertices. Since in total
  we have at most \(v(F)r/\epsilon\) many red arcs in \(\mathcal{T}(H')\), it
  follows that the depth of \(\mathcal{T}(H')\) is bounded by
  \begin{equation*}
    (v(F)r/\epsilon + 1)\abs{\mathcal{S}(F)}\ ,
  \end{equation*}
  and that consequently
  \begin{equation*}
    v\bigl(\mathcal{T}(H')\bigr) \leq 1 + r + r^2 + \dotsb +
    r^{(v(F)r/\epsilon + 1)\abs{\mathcal{S}(F)}} \leq r^{(v(F)r/\epsilon + 1)\abs{\mathcal{S}(F)} + 1}.
  \end{equation*}
  Since each vertex of \(\mathcal{T}(H')\) corresponds to at most \(v(F)\)
  vertices of \(H'\) we finally obtain that
  \begin{equation*}
    \begin{split}
      v(H') &\leq r^{(v(F)r/\epsilon + 1)\abs{\mathcal{S}(F)} + 1}\cdot v(F)
      \stackrel{~\eqref{eq:def-vmax}}{=} v_{\max}/r - 1.
    \end{split}
  \end{equation*}
\end{proof}

\begin{remark} The reader might wonder where exactly an attempt to extend the \emph{edge} Achlioptas lower bound proof in the same way fails. The issue arises with the definition of the graphs $K'_i$ that capture possible overlaps of the $r$-matched graphs $H_1', J_2', \ldots, J_r'$. In the edge case it is not possible to define these in such a way that the analogue of~\eqref{eq:mu-H'-1} holds.

\end{remark}

\section{The edge case}
\label{sec:edge-case}

In this section we prove Theorem~\ref{thm:1}, our separation result for the edge case.

As already mentioned, it is not hard to see that the Achlioptas game and the balanced Ramsey game have different
thresholds for certain forests. The simplest example is the case where $F$ is
the star with three rays and $r=2$: By the pigeon-hole principle, in the
balanced Ramsey game the player will lose the game as soon as the board
contains a star with five rays, which by a standard result a.a.s.\ happens
after $\Theta(n^{2-6/5})=\Theta(n^{4/5})$ many steps (see e.g.\ \cite[Section 3.1]{purple-book}). Thus the threshold of the
balanced Ramsey game is bounded from above by $n^{4/5}$.  In the Achlioptas
game with the same parameters on the other hand, stars on 5 edges are not an
issue, as typically the player can simply choose not to pick more than 2 edges
out of each such star. Specifically, the results of~\cite{Mutze2011}
yield a strictly higher threshold of \(n^{6/7}\) for the Achlioptas game.

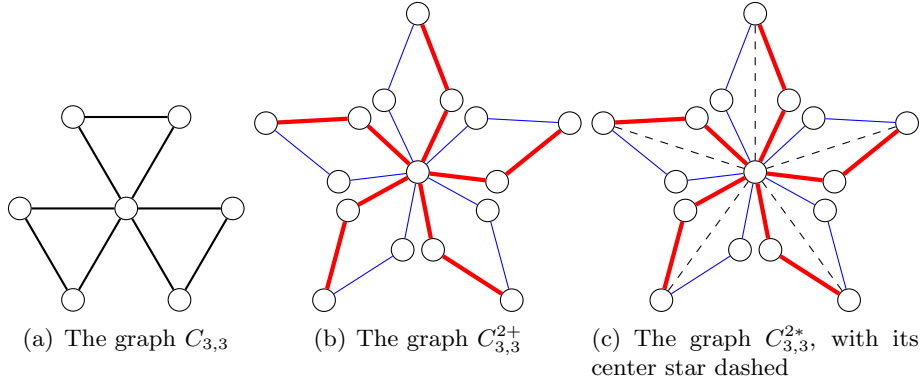
\begin{figure}
  \begin{center}
    \newcommand{\figurescale}{0.7}
    \subfigure[The graph \(C_{3,3}\)]{%
      \begin{tikzpicture}[scale=\figurescale,foo/.style={draw,shape=circle,inner sep=3pt}]
        \node[foo] (n0) at (0,0) {};
        \foreach \a in {90,210,330} {
          \node[foo] (n\a2) at (\a+30:2) {};
          \node[foo] (n\a3) at (\a-30:2) {};
          \path[draw,thick] (n0) -- (n\a2) -- (n\a3) -- (n0);
        };
      \end{tikzpicture}%
      \label{fig:ex:f}%
    }
    \subfigure[The graph \(C_{3,3}^{2+}\)]{%
      \begin{tikzpicture}[scale=\figurescale,foo/.style={draw,shape=circle,inner sep=3pt}]
        \node[foo] (n0) at (0,0) {};
        \foreach \a in {18,90,...,377} {
          \node[foo] (n\a2) at (\a-25:1.5) {};
          \node[foo] (n\a3) at (\a+25:1.5) {};
          \node[foo] (n\a4) at (\a:3) {};
          \path[draw,red,ultra thick] (n0) -- (n\a2) -- (n\a4);
          \path[draw,blue] (n\a4) -- (n\a3) -- (n0);
        };
      \end{tikzpicture}%
      \label{fig:ex:fstar}%
    }
    \subfigure[The graph \(C_{3,3}^{2*}\), with its center star dashed]{%
      \begin{tikzpicture}[scale=\figurescale,foo/.style={draw,shape=circle,inner sep=3pt}]
        \node[foo] (n0) at (0,0) {};
        \foreach \a in {18,90,...,377} {
          \node[foo] (n\a2) at (\a-25:1.5) {};
          \node[foo] (n\a3) at (\a+25:1.5) {};
          \node[foo] (n\a4) at (\a:3) {};
          \path[draw,red,ultra thick] (n0) -- (n\a2) -- (n\a4);
          \path[draw,blue] (n\a4) -- (n\a3) -- (n0);
          \path[draw,dashed] (n\a4) -- (n0);
        };
      \end{tikzpicture}%
      \label{fig:ex:fstarcolor}%
    }
  \end{center}
  \caption{A graph with different thresholds for the Achlioptas and the
    balanced Ramsey game.}
  \label{fig:ex}
\end{figure}

As it turns out, similar pigeon-hole problems as in the star example may arise for more complex graphs
as well. The simplest such example is given by the graph \(C_{3,3}\)
consisting of 3 triangles joined at one vertex, see Figure~\ref{fig:ex:f}. The
results of~\cite{Mutze2011} yield a threshold of \(n^{2-22/35} =
n^{1.371\dots}\) for the Achlioptas game with this graph and $r=2$. As we will see, the threshold of the corresponding balanced Ramsey game is at most \(n^{1.36}\). 
The reason is that, regardless of the strategy Painter uses, many copies of the graph  \(C_{3,3}^{2+}\) colored exactly as in Figure~\ref{fig:ex:fstar} will appear relatively early in the game. Once all 5 edges drawn dashed in
Figure~\ref{fig:ex:fstarcolor} have appeared in such a copy, by the pigeon-hole principle Painter will have created a monochromatic copy of $F=C_{3,3}$. As $C_{3,3}^{2*}$ has 16 vertices and 25 edges, the upper bound resulting from this argument is $n^{2-16/25} = n^{1.36}$.

We will show that this argument generalizes to any graph $F$ formed by some number of cycles of the same length joined at a common vertex, and to any number $r\geq 2$ of colors.

\begin{definition}
  Let \clk{} denote the graph obtained by joining \(k\) cycles of length
  \(\ell\) at one common vertex.
\end{definition}
We will prove:
\begin{theorem}\label{thm:separation}
  For all integers \(\ell\geq 3\), \(k \geq 3\), and $r\geq 2$, the threshold of the
  the balanced Ramsey game with parameters $C_{\ell,k}$ and $r$ is strictly lower than the threshold of the Achlioptas game with the same parameters.
\end{theorem}

We first give an upper bound on the threshold of the balanced Ramsey game with parameters $C_{\ell,k}$ and $r$. To do so we will use an offline result that is very similar and can be proved completely analogously to \cite[Theorem 15]{Krivelevich2010} for the Achlioptas case. For any graph  \(F\) with at least one edge, we let
  \begin{equation}\label{eq:def-m2}
    m_2(F) := \max_{H\subseteq F: v(H)\geq 3} \frac{e(H)-1}{v(H)-2}
  \end{equation}
  if $v(F)\geq 3$, and $m_2(F)=1/2$ otherwise (i.e., if $F=K_2$). By $G^r(n,m)$ we denote a random $r$-edge-matched graph obtained by sampling a random graph $G(n,m)$ on $n$ vertices with $m$ edges uniformly at random, and then partitioning the $m$ edges into sets of size $r$ uniformly at random (we assume that $m$ is divisible by $r$). Note that by symmetry the board of the edge Achlioptas or balanced Ramsey game after $m/r$ steps is distributed exactly like $G^r(n,m)$. A \emph{balanced coloring} of $G^r(n,m)$ is an edge-coloring that uses each of the $r$ available colors for exactly one edge in each $r$-set. Note that in the balanced Ramsey game, the goal is to find such a balanced coloring in an online setting. The following theorem concerns the same problem in an \emph{offline} setting.

\begin{theorem}\label{thm:m2-bound}
  Let \(F\) be a fixed graph with at least one edge, and let $c:E(F)\to\{1, \ldots, r\}$ be an arbitrary edge-coloring of $F$. There exist positive constants \(C=C(F,r)\) and \(a=a(F,r)\) such
  that for \(m \geq Cn^{2-1/m_2(F)}\) with \(m \ll n^2\), a.a.s.\ every balanced
  coloring of $G^r(n,m)$ contains at least \(a n^{v(F)}(m/n^2)^{e(F)}\) many
  copies of \(F\) colored as specified by \(c\).
  \qed
\end{theorem}


We now prove the desired upper bound on the balanced Ramsey threshold for the graphs $C_{\ell,k}$.

\begin{lemma} \label{lemma:ub-bal}
  For all integers \(\ell\geq 3\), \(k \geq 3\), and $r\geq 2$, the
  threshold for the balanced Ramsey game with parameters $\clk$ and
  \(r\)  is at most
  \begin{equation*}
    N_{\text{UB-bal}}(\ell,k,r,n) := n^{2-\frac{(r(\ell-2)+1)(r(k-1)+1)+1}{(r(\ell-1)+1)(r(k-1)+1)}}.
  \end{equation*}
\end{lemma}

\begin{proof}
  \newcommand{\dbal}{d_{\text{bal}}}
  Consider the graph obtained by joining one endpoint of \(r\) paths of length
  \(\ell-1\) in one common vertex and the other endpoint of each in a second
  common vertex. We call this graph a \emph{petal} and the two vertices in
  which all paths meet the \emph{endpoints} of the petal. We will refer to the non-edge connecting the two endpoints of a petal as the
  \emph{missing edge} of that petal.  Let \(\clkp\) denote the graph obtained by joining one endpoint of $k^*:=r(k-1)+1$ many petals at a common vertex. We say that a copy of \(\clkp\) on the game board is \emph{properly colored} if the two endpoints of each of its  petals are connected by a path (of length $\ell-1$) in each color. See Figure~\ref{fig:ex:fstar} for an example
  of a properly colored \(C_{3,3}^{2+}\). The \emph{center
    star} of a copy of \(\clkp\) is the graph obtained as the union of all
  missing edges of the petals of \(\clkp\). We denote with \(\clk^{r*}\) the
  union of \(\clkp\) and its center star, cfr.\ Figure~\ref{fig:ex:fstarcolor}. Clearly, we have
  \begin{equation*}
    \begin{split}
      e(\clk^{r*}) &= e(\clkp) + k^* = (r(\ell-1)+1)(r(k-1)+1),\\
      v(\clk^{r*}) &= v(\clkp) = (r(\ell-2)+1)(r(k-1)+1)+1.
    \end{split}
  \end{equation*}
  Let $d^*:=e(\clk^{r*})/v(\clk^{r*})$, and note that $N_{\text{UB-bal}}(\ell,k,r,n)= n^{2-1/d^*}$.

It is not hard to check that $m_2(\clkp) = (r(\ell-1)-1)/(r(\ell-2))$ (the maximum in~\eqref{eq:def-m2} is attained by a single petal of \clkp), and it is also quite straightforward to verify that this quantity is strictly less than $d^*$.

  Let now \(N \gg N_{\text{UB-bal}}=n^{2-1/d^*}\) with $N\ll n^2$ be given, and assume w.l.o.g.\
  that $N$ is even. Set \(p :=
  rN/n^2\). Observing that $N\gg n^{2-1/d^*}\geq n^{2-1/m_2(\clk^{r+})}$, we obtain with Theorem~\ref{thm:m2-bound} that a.a.s., after $N/2$ steps of the balanced Ramsey game the board  contains $a(\clkp,r)\cdot n^{v(\clkp)}p^{e(\clkp)}2^{-e(\clkp)}=:M'$ many properly colored copies of $\clkp$, regardless of Painter's strategy.  Furthermore, the expected number of copies of $\clkp$ (ignoring any coloring) in which at least one edge of the center star is already present after $N/2$ steps is $O(M'p)=o(M')$. It follows with Markov's inequality that after $N/2$ steps a.a.s.\ there are at least $M:=0.99M'$ properly colored copies of $\clkp$ such that in each of these, none of the edges of the center star is already present.

  Let \(\clkp \cup_{J} \clkp\) denote the union of two copies of \(\clkp\)
  which intersect in a graph \(J\) and whose (missing) center stars intersect
  in a nonempty graph \(J_S\). Further let \(J^* := J \cup J_S\). Let the random variable
  \(M_J\) denote the number of copies of \(\clkp \cup_{J} \clkp\) (ignoring any
  coloring) contained in the game board after the first \(N/2\) steps. We have
  \begin{equation}\label{eq:M-J-center-star-bound}
    \begin{split}
      \E[M_{J}] &=
			\Theta(n^{2v(\clkp)-v(J)}p^{2e(\clkp)-e(J)})
      \\&=\Theta(n^{2v(\clkp)}p^{2e(\clkp)})n^{-v(J)}p^{-e(J)}
      \\&= \Theta(M^2)n^{-v(J^*)}p^{-e(J^*)+e(J_S)},
    \end{split}
  \end{equation}
	where in the last step we used that $e(J^*)=e(J)+e(J_S)$.
	
  Note that \(\clks\) is a balanced graph, i.e.\ for all subgraphs \(H
  \subseteq \clks\) with \(v(H) \geq 1\) we have \(e(H)/v(H) \leq
  e(\clks)/v(\clks) = d^*\). This holds in particular also for \(H = J^*\). As \(p
  \gg n^{-1/d^*}\), it follows that $n^{v(J^*)}p^{e(J^*)}=\omega(1)$. Hence by Markov's inequality we obtain from
  \eqref{eq:M-J-center-star-bound} that a.a.s.
  \begin{equation}
    \label{eq:M-J-aas-bound}
    M_J = o(M^2)p^{e(J_S)}
  \end{equation}
(i.e., for an approriate function $f(n)=o(1)$ a.a.s.\ we have $M_J\leq f(n) M^2 p^{e(J_S)}$).

  For the remaining $N/2$ steps of the game we condition on having at least~$M$ properly colored copies of $\clkp$ whose center star edges are not already present, and on 
  \(M_{J}\) being as above for all \(J \subseteq \clkp\). As the number of graphs $J$ is a constant depending only on $k$, $\ell$, and $r$, a.a.s\ all these properties hold simultaneously after $N/2$ steps. Using the second moment method, we will show that in the remaining $N/2$ steps,  a.a.s.\ in at least one of the properly colored copies of \(\clkp\) all edges of the center star will appear. Clearly, this then forces Painter to complete
  a monochromatic copy of $C_{\ell,k}$ by the pigeon-hole principle.

Fix a family of exactly $M$ properly colored
  copies of \(\clkp\) (say the lexicographally first ones; w.l.o.g.\ $M$ is an integer), and
  let \(S_1, \dotsc, S_M\) denote the (not necessarily distinct) center stars of these copies. For each $S_i$ let $Z_i$
  denote the indicator random variable for the event that the \(k^*\) edges of $S_i$  will appear in the remaining $N/2$ steps of the game. Let \(Z\)
  denote the sum over all \(Z_i\). As the $rN/2$ random edges revealed in the second half of the game are distributed uniformly among the $\binom{n}{2}-rN/2$ edges never seen before, we have
	    \begin{equation*}
      \E[Z_i]  =\frac{\binom{\binom{n}{2}-rN/2-k^*}{rN/2-k^*}}{\binom{\binom{n}{2}-rN/2}{rN/2}}= \Theta(p^{k^*})
    \end{equation*}
for all $i$, and hence
    \begin{equation*}
      \E[Z] =  \Theta(M p^{k^*}) = \Theta(n^{v(\clkp)}p^{e(\clkp)+k^*})=				
			\Theta\bigl(n^{v(\clk^{r*})}p^{e(\clk^{r*})}\bigr).
    \end{equation*}
    By our choice of \(N\) this quantity is \(\omega(1)\).

  It remains to establish
  concentration of $Z$ via the second moment method --- it then follows that $Z\geq 1$ a.a.s., which as discussed implies that Painter loses the game.
  We have
  \begin{equation*}
    \begin{split}
      \Var[Z] &= \sum_{i,j = 1}^{M}\bigl(\E[Z_iZ_j]-\E[Z_i]\E[Z_j]\bigr) \leq
      \sum_{\substack{J \subseteq \clkp\\e(J_S) \geq 1}} M_J\cdot \Theta(p^{2k^*- e(J_S)})\\
      &\stackrel{\eqref{eq:M-J-aas-bound}}{=} \sum_{\substack{J \subseteq
          \clkp\\e(J_S) \geq 1}} o(M^2)p^{2k^*} = o(\E[Z]^2).
    \end{split}
  \end{equation*}
  The last equality follows from the fact that the number of possible choices
  for \(J\) is a constant depending only on \(k, \ell\) and \(r\). This
  concludes the proof.
\end{proof}

We conclude the proof of Theorem~\ref{thm:separation} by deriving a lower bound on the Achlioptas threshold for the graphs $C_{\ell,k}$ from the general formula given in~\cite{Mutze2011}.

\begin{lemma} \label{lemma:lb-achl}
  For all integers \(\ell\geq 3\), \(k \geq 3\), and $r\geq 2$, the
  threshold for the Achlioptas game with parameters $\clk$ and
  \(r\)  is at least
  \begin{equation*}
    N_{\text{LB-Achl}}(\ell,k,r,n) = n^{2-\frac{(r(l-2)+1)(r^k-1)+r-1}{(r(l-1)+1)(r^k-1)}}.
  \end{equation*}\end{lemma}

\begin{proof}
For the reader's convenience we reproduce the general edge Achlioptas threshold formula here. For notational details we refer to~\cite{Mutze2011}.

For any nonempty edge-ordered graph $(H_1,\pi)$, $\pi=(e_1,\dots,e_h)$, any sequence of subgraphs $H_2,\dots,H_h\subseteq H_1$ with $H_i\subseteq H_1\setminus\{e_1,\dots,e_{i-1}\}$ and $e_i\in H_i$ for all $2\leq i\leq h$, and any integer $r\geq 2$, define coefficients $c_i=c_i((H_1, \pi), H_2, \dots,H_h,r)$ recursively by
\begin{equation*}
\begin{split}
c_1 &:= r\enspace,\\
c_i &:= (r-1)\cdot\sum_{j=1}^{i-1}c_j\ind{e_i \in H_j}\enspace, \quad 2\leq i \leq h\enspace,
\end{split}
\end{equation*}
(where \(\ind{e_i \in H_j}=1\) if \(e_i \in H_j\) and \(\ind{e_i \in H_j}=0\)
otherwise), and set
\begin{equation} \label{eq:def-drstar-edge}
d^{r*}(H_1, \pi) := \max_{\substack{H_2, \dots, H_h\\ \forall i\geq 2:\; H_i\subseteq H_1\setminus\{e_1,\dots,e_{i-1}\} \;\wedge\; e_i\in H_i}} \frac{1+\sum_{i=1}^h c_i(e(H_i)-1)}{2+\sum_{i=1}^h c_i(v(H_i)-2)}\enspace.
\end{equation}
Furthermore, we set for any integer $r\geq 2$ and any nonempty graph $F$
\begin{equation*}
m^{r*}(F) := \min_{\pi\in\Pi(E(F))}\max_{H_1\subseteq F} d^{r*}(H_1, \pi|_{H_1})\enspace.
\end{equation*}
  The threshold of the Achlioptas game with parameters $F$ and $r$ is then given by \(N_{0}(F,r,n) = n^{2-1/m^{r*}(F)}\).

  We now prove that $m^{r*}(\clk)$ is bounded from below as claimed in the lemma.    Let
  \(\pi = (e_1, \dotsc, e_h)\) be an arbitrary permutation of the edges of
  \clk. Denote
  with \(e_{t_1}, \dotsc, e_{t_k}\) the first edge in each of the $k$ cycles of $\clk$ according to $\pi$, in order of their appearance in $\pi$. (Thus in particular \(e_{t_1} = e_1\).)
  Let \(C_1, \dotsc, C_k\) the corresponding cycles in $\clk$, i.e. $e_{t_i}\in C_i$ for all $i$. Choose
  \begin{equation*}
    H_i =
    \begin{cases}
      e_i & i \notin \{t_1, \dotsc, t_k\}\\
      \bigcup_{j \geq i}^k C_j & i = t_j
    \end{cases}.
  \end{equation*}
  Note that this choice is compatible with the requirements of
  \eqref{eq:def-drstar-edge}. This yields
  \begin{equation*}
    \begin{split}
      e(H_{t_j}) &= (k-j+1)\ell\\
      v(H_{t_j}) &= (k-j+1)(\ell-1)+1,
    \end{split}
  \end{equation*}
  and
  \begin{equation*}
    c_{t_j} =
    \begin{cases}
      r & j = 1\\
      (r-1)r^{j-1} & j \neq 1.
    \end{cases}
  \end{equation*}
  Note that the coefficients \(c_i\) for \(i \notin \{t_1,\dotsc, t_k\}\) are not
  required, as both \(e(H_i)-1\) and \(v(H_i)-2\) are 0. It is somewhat
  tedious but straightforward to verify that
  \begin{equation*}
    \frac{1 + \sum_{i = 1}^{h}c_i\bigl(e(H_i)-1\bigr)}{2 + \sum_{i = 1}^{h}c_i
      \bigl(v(H_i)-2\bigr)} = \frac{(r(l-1)+1)(r^k-1)}{(r(l-2)+1)(r^k-1)+r-1}.
  \end{equation*}
  As this holds for any edge ordering \(\pi\in\Pi(E(C_{\ell,k}))\), we readily obtain the desired lower bound
  \begin{equation*}
    m^{r*}(C_{\ell,k}) = \min_{\pi \in \Pi(E(C_{\ell,k}))}\max_{H_1\subseteq C_{\ell,k}} \drs(H_1, \pi|_{H_1}) \geq \frac{(r(l-1)+1)(r^k-1)}{(r(l-2)+1)(r^k-1)+r-1}.
  \end{equation*}
\end{proof}

Theorem~\ref{thm:separation} now follows, after some calculation, from Lemmas~\ref{lemma:ub-bal} and~\ref{lemma:lb-achl}.

\subsection*{Acknowledgment}
The authors thank the anonymous referee for the thorough and helpful comments concerning the presentation of this work.

\bibliographystyle{plain}
\bibliography{biblio}

\newpage
\section{Appendix}
\label{sec:appendix}

\subsection{Upper bound for the Achlioptas game}
\label{sec:upper-bound}

In this section we prove Theorem~\ref{thm:achlioptas-ubound}. The proof here
is an adaptation to the vertex case of the corresponding edge-case proof in
\cite{Mutze2011}. We make use of Theorem~\ref{thm:threshold-equivalence} (ignoring the last sentence in its statement; this is restated and proved as Lemma~\ref{lem:threshold-equivalence} below) and a
technical lemma (Lemma~\ref{lem:subsets-of-Fpi-are-not-dense} below), both of which are proved in the next section.

Before we start we wish to present the following adaptation to \(r\)-matched
graphs of Bollobás' classical small subgraphs result~\cite{Bollobas1981}.
\begin{theorem}\label{thm:bollobas-r-matched}
  Let \(r\geq 2\) be a fixed integer, and let \(F\) be a fixed \(r\)-matched graph
  with at least one edge. Define
  \begin{equation*}
    m^r(F) := \max_{\substack{H\subseteq F:\\\kappa(H) > 0}} \frac{e(H)}{\kappa(H)}.
  \end{equation*}
  Then the threshold for the appearance of\/ \(F\) in \(G^r(n,p)\) is
  \begin{equation*}
    p_{0}(n) = n^{-1/m^r(F)}.
  \end{equation*}
  Further, if \(p \gg n^{-1/m^r(F)}\) we have that the number of copies of
  \(F\) in \(G^r(n,p)\) is a.a.s.\
  \begin{equation*}
    \Theta(n^{\kappa(F)}p^{e(F)}).
  \end{equation*}
\end{theorem}
One can prove this by an easy application of the first and second moment method.
We do not require this result, but it is useful to gain a better intuition for
our proof. Note that we could state the first part of Theorem~\ref{thm:bollobas-r-matched} equivalently as follows. 
\begin{theorem}\label{thm:bollobas-r-matched-mu}
  Let \(r\geq 2\) be a fixed integer, and let \(F\) be a fixed \(r\)-matched graph
  with at least one edge. Let \(\theta' = \theta'(F,r)\) be the unique solution of
  \begin{equation*}
    \min_{H \subseteq F}\mr(H) \stackrel{!}{=} 0,
  \end{equation*}
  where \(\mr\) is defined in \eqref{eq:def-mu-r-theta}. Then the threshold
  for the appearance of\/ \(F\) in \(G^r(n,p)\) is
  \begin{equation*}
    p_{0}(n) = n^{-\theta'}.
  \end{equation*}
\end{theorem}
Concerning the second part of Theorem~\ref{thm:bollobas-r-matched}, recall also that for $p=n^{-\theta}$  we have $n^{\kappa(F)}p^{e(F)}=n^{\mr(F)}$. The two ``dual'' formulations of our threshold result in Theorem~\ref{thm:achlioptas-ubound} and (the last sentence of) Theorem~\ref{thm:threshold-equivalence} are related to each other similarly as the two statements above.

In order to prove Theorem~\ref{thm:achlioptas-ubound} it is not sufficient to
consider only a graph \(F\), we additionally need to consider the order in
which its vertices are presented to the player. In our proof this is encoded
by an ordered graph \((F, \pi)\). Recall that when we use terms such as first
or last for the vertices of \(F\) we mean this with respect to the order in
which they are presented to the player. In that context, for \(\pi =
(u_1,\dotsc, u_f)\) the last vertex is \(u_1\) and the first \(u_f\).

If as an adversary we wanted to force the player to create a copy of
\(F\), we would wish to be able to present \(r\) copies of \(F_{-}\) (\(F\)
without the last vertex), an additional \(r\)-set, and edges such that choosing any of
the vertices in the \(r\)-set completes a copy of \(F_{-}\) to a copy of
\(F\). In such a situation the player has no choice and loses. Of course the
player could try to avoid creating copies of \(F_{-}\) in the first place,
so that this situation does not arise. However, applying the same argument
recursively, we could force the creation of \(r\) copies of \(F_{-}\) by
\(r^2\) copies of \(F_{2-}\) (\(F\) missing the last 2 vertices), \(r\)
many \(r\)-sets, and all edges necessary to join each of the \(r^2\) vertices
in the \(r\)-sets to a different copy of \(F_{2-}\), in such a way as to form
\(r^2\) many ``threats'' for the player. The player must choose one vertex in
all of these \(r\) many \(r\)-sets and is thus forced to create \(r\) copies
of \(F_{-}\).
We can continue this reasoning recursively until at the tail end we
have \(r^{v(F)-1}\) disjoint \(r\)-sets, out of each of which the player is
forced to choose 1 vertex.
Assuming that such a recursive ``history graph'' for \(F\)
appears in the game in the correct order, it would guarantee that the player is
left no choice but to create a copy of \(F\). We denote such a construct with
\(\Fpi\) and formalize its definition below.

In the following, a \emph{grey-black \(r\)-matched graph} is a tuple
\(H = (V, E, \mathcal{K}, B)\), where \((V, E, \mathcal{K})\) is an
\(r\)-matched graph, and \(B\) is a set of vertices containing exactly one
vertex from every \(r\)-set in \(\mathcal{K}\). We interpret \(B\) as the set
of vertices chosen by the player during the game, and call them the
black vertices. The remaining \(\abs{\mathcal{K}}(r-1)\) vertices are considered
grey (to indicate ``presented but not chosen'').
Sometimes we ignore the coloring and tacitly
identify \(H\) with the underlying \(r\)-matched graph \((V, E,
\mathcal{K})\).

Recall that the board of the game is distributed as a random \(r\)-matched
graph \(\Grnp\), and that we defined its state after \(1 \leq i \leq n/r\)
rounds with \(G_{i}\). For the purpose of this section we additionally require
information about which vertices were chosen by the player. To this end, for
each \(1 \leq i \leq n/r\), in this section we append to \(G_{i}\) the set
\(B_i\) of vertices chosen by the player up to round \(i\) and consider
\(G_{i}\) to be a grey-black \(r\)-matched graph.

\begin{definition}
 \label{def:grey-black-F-pi}
 Let \((F, \pi)\) be an ordered graph with \(\pi = (u_1, \dotsc, u_f)\). Then
 we define the \emph{grey-black} \(r\)-matched graph \(\Fpi\) and
 a distinguished black copy of \((F,\pi)\), the \emph{central copy of \((F,\pi)\) in \(\Fpi\)},
 recursively as follows
 \begin{itemize}
 \item If \(v(F) = 1\), then \(\Fpi\) consists of\/ one \(r\)-set with a
   distinguished black vertex. This vertex is the central copy of \(F\) in
   \(\Fpi\).
 \item If \(v(F) \neq 1\), then \(\Fpi\) consists of the disjoint union of\/
   \(r\) copies of \((F\setminus{u_1})^{\pi\setminus u_1}_r\), denoted
   \(F^{\pi}_{-,1}, \dotsc, F^{\pi}_{-,r}\), an additional \(r\)-set \((v_1, \dotsc,
   v_r)\), and \(r\deg_F(u_1)\) many additional edges which for all \(1\leq
   i\leq r\) connect \(v_i\) to \(F^\pi_{-,i}\) and extend the central copy of
   \((F\setminus u_1, \pi\setminus u_1)\) in \(F^\pi_{-,i}\) to a copy of
   \((F,\pi)\). The vertex \(v_1\) is chosen as black and the copy of
   \((F,\pi)\) containing it is the central copy of \((F,\pi)\) in \(\Fpi\).
 \end{itemize}
 We refer to the additional \(r\)-set in the recursive step as the
 \emph{central \(r\)-set} of \(\Fpi\).
\end{definition}

As explained above, if the \(r\)-sets of a copy of \(\Fpi\) are
presented to the player in an ordering such that all \(r\)-sets deeper in the
recursion are presented before those at lower recursion depths, then the
player is forced to create a copy of~\(F\). As it turns out, the threshold
for this to happen in the game coincides with the threshold for the appearance of $\Fpi$ in the random $r$-matched graph $\Grnp$ (as an $r$-matched graph without any ordering or coloring). This last threshold is given by
Theorem~\ref{thm:bollobas-r-matched}.

Note that the ordering \(\pi\) on the vertices of \(F\) is crucial. For
different choices of \(\pi\) the corresponding grey-black \(r\)-matched graphs
\(\Fpi\) may have very different thresholds for their appearance in $\Grnp$. As the player has no influence over the order in which \(r\)-sets are
presented to her, the threshold for the game is bounded from above by, and indeed coincides with, the \emph{lowest} threshold for the appearance of \(\Fpi\) over all choices of \(\pi\). I.e., the threshold stated in Theorems~\ref{thm:achlioptas-ubound} and~\ref{thm:threshold-equivalence} can alternatively be written as $$p_0(F,r,n)=\min_{\pi\in \Pi(V(F))} n^{-1/m^r(\Fpi)},$$ where
\(m^r(\Fpi)\) is as defined in Theorem~\ref{thm:bollobas-r-matched}.

At a high level, the proof is an induction over \(v(F)\) and mirrors the
recursive definition of \(\Fpi\). At each step of the induction we divide the
\(r\)-sets presented to the player in 2 halves. We let the player play on the
first half and by induction we know that a.a.s.\ she must have created many copies of
\((F\setminus u_1)^{\pi\setminus u_1}\) (in the notation of
Definition~\ref{def:grey-black-F-pi}). Then we let the player play on the
second half of the \(r\)-sets and argue via first and second moment method that, conditional on a ``good'' first round,
a.a.s.\ enough \(r\)-sets presented in the second round are connected to \(r\)
copies of \((F\setminus u_1)^{\pi\setminus u_1}\) as in
Definition~\ref{def:grey-black-F-pi}.

To apply the second moment method we need the following lemma, which
essentially states that for \(p \gg n^{-\theta'}\), where \(\theta'\) is
defined below, the expected number of copies in \(G^r(n,p)\) of any subgraph
of \(\Fpi\) is \(\omega(1)\), cf.\ the remark after  Theorem~\ref{thm:bollobas-r-matched-mu}.
\begin{lemma}\label{lem:subsets-of-Fpi-are-not-dense}
 Let \(r \geq 2\) be an integer, and let \((F,\pi)\) be a nonempty ordered
 graph. Let \(\Fpi\) be as in Definition~\ref{def:grey-black-F-pi}, and
 let \(\theta' = \theta'(F,\pi,r)\) be the unique solution of
 \begin{equation}\label{eq:def-theta-prime-fixed-pi}
   \min_{H\subseteq F}\lr(H,\pi|_{H}) \stackrel{!}{=} 0,
 \end{equation}
 where \(\lr()\) is defined in \eqref{eq:lambda-r}. Then every \(r\)-matched
 subgraph \(J \subseteq \Fpi\) satisfies
 \begin{equation}\label{eq:mu-subgraph-positive}
   \mu_{r,\theta'}(J)\geq 0,
 \end{equation}
 where \(\mu_{r,\theta'}()\) is defined in~\eqref{eq:def-mu-r-theta}.
\end{lemma}
The proof of this lemma is long and technical, and therefore postponed to the next subsection.

The next lemma implements the inductive proof strategy outlined above. The
parameter \(t\) ensures that we can require inductively that \(r\cdot t\)
copies of e.g.\ \((F\setminus u_1)_r^{\pi\setminus u_1}\) evolve into \(t\)
copies of \(\Fpi\).

\begin{lemma}
 \label{lem:upper-bound-iteration}
 Let \(r\geq 2\) be an integer, and let \((F,\pi)\) be a nonempty ordered
 graph. Let \(t \geq 1\) be an integer, and let \(\mathcal{F}^\pi_r := t\cdot
 \Fpi\) denote the disjoint union of\/ \(t\) copies of \(\Fpi\). If \(1\gg p
 \gg n^{-\theta'}\), where \(\theta' = \theta'(F,\pi,r)\) is the unique
 solution of
 \begin{equation}
   \min_{H\subseteq F}\lr(H,\pi|_{H}) \stackrel{!}{=} 0\, ,
 \end{equation}
 and \(\lr()\) is defined in \eqref{eq:lambda-r}, then a.a.s.\ the number of
 copies of \(\mathcal{F}_r^\pi\) (as a grey-black \(r\)-matched
 graph) in \(G_{n/r}\)  is
 \begin{equation}
   \label{eq:upper-bound-lemma-theta}
   \Omega\bigl(n^{\kappa(\mathcal{F}_r^\pi)}p^{e(\mathcal{F}_r^\pi)}\bigr)
 \end{equation}
 regardless of the strategy of the player.
\end{lemma}

Before we prove this lemma, we show how it implies Theorem~\ref{thm:achlioptas-ubound}.

\begin{proof}[Proof of Theorem~\ref{thm:achlioptas-ubound}]
  By the equivalence stated in Theorem~\ref{thm:threshold-equivalence} (and proved in
  Lemma~\ref{lem:threshold-equivalence} below), it suffices to prove that for $p\gg n^{-\theta^*(F,r)}$, the player will a.a.s.\ create a copy of $F$ no matter how she plays. Let \(\pi \in \Pi(V(F))\) be an
  ordering maximizing the right hand side of~\eqref{eq:Lambda-r-theta} for $\theta= \theta^*(F,r)$, such that $\theta^*(F,r)= \theta'(F,\pi,r)$ for $\theta'$  as in
  Lemmas~\ref{lem:subsets-of-Fpi-are-not-dense} and
  \ref{lem:upper-bound-iteration}. Applying Lemma~\ref{lem:upper-bound-iteration} for
  \(t = 1\), we obtain that \(G_{n/r}\) a.a.s.\
  contains
  \begin{equation*}
    \Omega(n^{\kappa(\Fpi)}p^{e(\Fpi)}) \gg n^{\mu_{r,\theta^*}(F)}\geq 1,
  \end{equation*}
  many copies of \(\Fpi\), where the last inequality follows from Lemma~\ref{lem:subsets-of-Fpi-are-not-dense}. The central copy of \((F,\pi)\) in each of these
  copies of \(\Fpi\) is black, i.e.\ all its vertices were selected by the
  player. 
\end{proof}

We now prove Lemma~\ref{lem:upper-bound-iteration}, giving the main inductive argument of our upper bound proof.

\begin{proof}[Proof of Lemma~\ref{lem:upper-bound-iteration}]
 \newcommand{\Fonemin}{\mathcal{F}^{\pi}_{-}}
 We prove this lemma by induction on \(v(F)\) using the second moment
 method.

 To simplify the notation we drop all subscripts \(r\) from \(F^\pi_r\) and
 \(\mathcal{F}^\pi_r\).

 As a base case for the induction we consider the case of an empty \(F\). The
 lemma does not apply to this case directly (as \(\theta'\) is not
 well-defined), but a statement equivalent to
 \eqref{eq:upper-bound-lemma-theta} still holds and is all that we require for
 the induction.  If \(F\) contains no edges then \(\mathcal{F}^\pi\) contains no edges either and consists only of \(\kappa(\mathcal{F}^\pi)\) many disjoint \(r\)-sets. It trivially holds
 that \(G_{n/r}\) contains \(\Theta(n^{\kappa(\mathcal{F}^\pi)})\) many copies
 of \(\mathcal{F}^\pi\), regardless of the choice of \(p\).

 To discuss the induction step we first introduce some notation. Let \(\pi =
 (u_1, \dotsc, u_f)\), \(\pi_- = \pi \setminus u_1\) and \(F_{-} = F\setminus u_1\).
 Further let \(F_{-}^\pi\) denote the
 grey-black \(r\)-matched graph $(F_-)^{\pi_-}$, and denote by \(\Fonemin\) the
 disjoint union of \(rt\) copies of \(F^\pi_{-}\).

 We use a two-round approach for the induction step. In the first round we let
 the player make her choices for all \(r\)-sets in \(G_{n/(2r)}\). By the induction
 hypothesis we obtain a lower bound on the number of copies of \(\Fonemin\) that the
 player must have created which holds with high probability. Conditioning on
 the fact that the bound from the first round holds, we then derive a bound for
 the number of copies of \(\mathcal{F}^\pi\) which the player is forced to
 create when she is presented the remaining \(r\)-sets in \(G_{n/r}\).

 Note that if \(F_{-}\) is nonempty, then \(\theta'(F_-, \pi_-, r) \geq
 \theta'(F,\pi,r)\) (cf.~\eqref{eq:def-theta-prime-fixed-pi}), and we can
 apply the induction hypothesis for \(t \leftarrow r\cdot t\) and $(F,\pi) \leftarrow ( F_-, \pi_-)$ to \(G_{n/(2r)}\). If
 \(F_{-}\) is empty we apply the base case of the induction described
 above.

 We have that a.a.s.\ at least
 \begin{equation}
   \label{eq:num-N}
   N := cn^{\kappa(\Fonemin)}p^{e(\Fonemin)}
 \end{equation}
copies of $\Fonemin$ are created in the first round for some appropriate constant $c>0$. For the second round we condition on this
 event (and also on~\eqref{eq:M-J-prime-bound} below, which is however irrelevant for the time being). We fix a set of exactly $N$ copies of $\Fonemin$ (say the $N$ lexicographically first ones), and only consider these throughout the following.

 Recall that by the construction given in Definition~\ref{def:grey-black-F-pi}
 we can extend \(r\) copies of \(F^\pi_-\) to one copy of \(F^\pi\). To do so
 we add one new \(r\)-set and \(r \deg_F(u_1)\) edges (\(u_1\) is the last
 vertex of \((F,\pi)\)). Each of the \(r\) central copies of \((F_-,\pi_-)\) is
 connected to a different vertex of the \(r\)-set by \(\deg_F(u_1)\) edges and
 becomes a copy of \((F,\pi)\).

 By repeating the above \(t\) times in parallel, any \(t\) disjoint \(r\)-sets
 presented in the second round together with one of the \(N\) copies of \(\Fonemin\) can be
 extended to a copy of \(\mathcal{F}^\pi\), provided that the required edges
 appear in the second round of the game.

 Let \(M\) be the number of possible pairs of \(t\) disjoint \(r\)-sets and
 one copy of \(\Fonemin\). We index these pairs with $i=1, \ldots, M$. For each such pair there may be several possible edge sets which extend the pair to a copy of \(\mathcal{F}^\pi\) as described. We fix one arbitrarily and denote this edge set by $T_i$. We denote with \(\mathcal{F}^\pi_i\) the copy
 of \(\mathcal{F}^\pi\) that is created if all edges of $T_i$ appear during the second round. Note that $|T_i|=t\cdot r\deg_F(u_1)$ for all $i$. By \(\mathcal{K}_i\) we denote the family of \(t\) disjoint \(r\)-sets that belongs to pair $i$. Note that each such family belongs to $N$ pairs in total. 

 There are \(\Theta(n^t)\) possible ways to choose
 the \(r\)-sets, so we have
 \begin{equation}
   \label{eq:num-M}
   M = \Theta\bigl(n^{t}\bigr) \cdot N\stackrel{(\ref{eq:num-N})}{=} \Theta\bigl(n^{t +
     \kappa(\Fonemin)}p^{e(\Fonemin)}\bigr) =
   \Theta\bigl(n^{\kappa(\mathcal{F}^\pi)}p^{e(\Fonemin)}\bigr).
 \end{equation}

 For \(i = 1, \dotsc, M\) we define the indicator variable \(Z_i\) for the
 event that \(T_i\) is contained in \(G_{n/r}\). Set \(Z = \sum_{i=1}^M
 Z_i\). Note that $Z$ is a lower bound on the number of copies of $\mathcal{F}^\pi$ created during the second round.

 For each \(Z_i\) we have
 \begin{equation}
   \label{eq:e-z_i}
   \E[Z_i] = \P[Z_i = 1] = p^{rt\deg_F(u_1)}.
 \end{equation}
 For the expected value of \(Z\), conditioned on (\ref{eq:num-N}), we thus obtain
 \begin{equation}
   \label{eq:upper-bound-exp-z}
     \E[Z] = \sum_{i = 1}^M \E[Z_i] \stackrel{(\ref{eq:e-z_i})}{=}
     Mp^{rt\deg_F(u_1)} \stackrel{(\ref{eq:num-M})}{=} \Theta\bigl(n^{\kappa(\mathcal{F}^\pi)}p^{e(\mathcal{F}^\pi)}\bigr).
 \end{equation}

 To apply the second moment method, we need to bound the variance of~\(Z\).
 Denote by \(I \subseteq \{1,\dotsc, M\}^2\) the set of pairs of indices
 \((i,j)\) such that \(T_i \cap T_j \neq \emptyset\). For \((i,j) \in I\) let
 \(\kappa_{ij} = \mathcal{K}_i \cap \mathcal{K}_j\) and \(t_{ij} = |T_i\cap
 T_j|\). For such pairs of indices we have
 \begin{equation}
   \label{eq:ezizj}
   \E[Z_iZ_j] = p^{2rt\deg_F(u_1) - t_{ij}}.
 \end{equation}
 For indices \(i,j\) with \(T_i \cap T_j = \emptyset\) on the other hand
 \(Z_i\) and \(Z_j\) are independent and can be dropped from the variance
 calculation. We obtain
 \begin{equation}
     \label{eq:upper-bound-var-1}
     \begin{split}
       \Var[Z] &= \sum_{i,j = 1}^M \Bigl(\E[Z_i Z_j] - \E[Z_i]\E[Z_j]\Bigr)
       \leq \sum_{(i,j)\in I} \E[Z_iZ_j] \\&= \sum_{(i,j)\in I} p^{2rt\deg_F(u_1)
         - t_{ij}}.
     \end{split}
 \end{equation}

 Let \(\mathcal{J} \subseteq \mathcal{F}^\pi\) be a subgraph that contains at
 least one of the \(t\) central \(r\)-sets of the \(t\) copies of \(F^\pi\) in
 \(\mathcal{F}^\pi\). Denote with \(\mathcal{K}_\mathcal{J}\) the family of these
 central \(r\)-sets in \(\mathcal{J}\), and with \(\kappa_\mathcal{J} := \abs{\mathcal{K}_\mathcal{J}}\) their
 number. Let \(\mathcal{J}_-\) be the \(r\)-matched graph obtained from
 \(\mathcal{J}\) by removing \(\mathcal{K}_{\mathcal{J}}\) and all incident edges. Let \(T_{\mathcal{J}}\) be the graph induced by
 the edges of \(\mathcal{J}\) between \(\mathcal{K}_\mathcal{J}\) and \(\mathcal{J}_-\) .

 Let  \(M_{\mathcal{J}}\) denote the number of pairs \((i,j)\) for which the intersection of \(\mathcal{F}^\pi_i\) and  \(\mathcal{F}^\pi_j\) is isomorphic to $\mathcal{J}$.
Note that then \(t_{ij} =
 e(T_{\mathcal{J}})\) and $|\mathcal{K}_i\cup\mathcal{K}_j|
=2t-\kappa_{\mathcal{J}}$. 

We will bound \(M_{\mathcal{J}}\) by the number of
 (uncolored) copies of \(\mathcal{F}^\pi_{-} \cup_{\mathcal{J}_-}
 \mathcal{F}^\pi_-\) which are created in the first round times the
 \(\Theta(n^{2t - \kappa_{\mathcal{J}}})\) choices for
$\mathcal{K}_i$ and $\mathcal{K}_j$ from all $r$-sets of the second round. Here \(\mathcal{F}^\pi_{-}
 \cup_{\mathcal{J}_-} \mathcal{F}^\pi_-\) denotes an uncolored \(r\)-matched
 graph formed by the union of two copies of \(\mathcal{F}^\pi_{-}\) which
 intersect in \(\mathcal{J_-}\).

 Let thus $M'_\mathcal{J}$ denote the number of copies of \(\mathcal{F}^\pi_{-}
 \cup_{\mathcal{J}_-} \mathcal{F}^\pi_-\) contained in $G_{n/(2r)}$, multiplied with the number of choices for
$\mathcal{K}_i$ and $\mathcal{K}_j$ from the $r$-sets of the second round. Note that $M'_\mathcal{J}$ is a random variable that depends only on the edges of the first round, and that $M_\mathcal{J}\leq M'_\mathcal{J}$.
We have
 \begin{equation} \label{eq:M-J-prime-p}
   \begin{split}
     \E[M_\mathcal{J}'] &= \Theta(n^{2\kappa(\mathcal{F}^\pi_-) -
       \kappa(\mathcal{J}_-)}p^{2e(\mathcal{F}^\pi_{-}) - e(\mathcal{J}_-)})\cdot \Theta(n^{2t - \kappa_\mathcal{J}}) \\
     &= \Theta(n^{2\kappa(\mathcal{F}^\pi)}p^{2e(\mathcal{F}^\pi_{-})})n^{-\kappa(\mathcal{J}_-)-\kappa_\mathcal{J}}p^{-
       e(\mathcal{J}_-)} \stackrel{(\ref{eq:num-M})}{=} \Theta(M^2)n^{-\kappa(\mathcal{J})}p^{-e(\mathcal{J}_-)}\\
			&=\Theta(M^2)n^{-\kappa(\mathcal{J})}p^{-e(\mathcal{J})+e(T_\mathcal{J})}.
   \end{split}
 \end{equation}
 As \(\mathcal{F}^\pi\) consists of \(t\) disjoint copies of \(F^\pi\) we can
 apply Lemma~\ref{lem:subsets-of-Fpi-are-not-dense} once for each intersection
 of \(\mathcal{J}\subseteq\mathcal{F}^\pi\) with one of the copies of \(F^\pi\). For every such
 intersection \(J \subseteq F^\pi\), as \(p \gg n^{-\theta'}\), we have by
 Lemma~\ref{lem:subsets-of-Fpi-are-not-dense} that
 \begin{equation*}
   n^{-\kappa(J)}p^{-e(J)} \ll n^{-\kappa(J) + \theta e(J)} \stackrel{\eqref{eq:def-mu-r-theta}}{=} n^{-\mu_{r,\theta}(J)} = O(1).
 \end{equation*}
 As \(t\) is a fixed constant the same holds if we replace \(J\) by
 \(\mathcal{J}\). Together with \eqref{eq:M-J-prime-p} and Markov's inequality
 this implies that a.a.s.
 \begin{equation}\label{eq:M-J-prime-bound}
   M_{\mathcal{J}}' \ll M^2p^{e(T_\mathcal{J})}
 \end{equation}
(i.e., for an approriate function $f(n)=o(1)$ a.a.s.\ we have $M_{\mathcal{J}}'\leq f(n) M^2 p^{e(T_\mathcal{J})}$).
 As
 the number of ways of choosing \(\mathcal{J}\subseteq \mathcal{F}^\pi\) is a
 constant depending only on \(F\), \(r\) and \(\pi\),
 \eqref{eq:M-J-prime-bound} holds a.a.s.\ for every possible choice of
 \(\mathcal{J}\) simultaneously. For the second round we condition on the
 first one satisfying \eqref{eq:num-N} and \eqref{eq:M-J-prime-bound} for all \(\mathcal{J}\subseteq\mathcal{F}^\pi\). With
 this we obtain from \eqref{eq:upper-bound-var-1} that
 \begin{equation*}
   \begin{split}
     \Var[Z] &= \sum_{(i,j)\in I} p^{2rt\deg_F(u_1) - t_{ij}} =
     \sum_{\substack{\mathcal{J}\subseteq \mathcal{F}^\pi:\\ \kappa_\mathcal{J}\geq 1}} M_{\mathcal{J}}
     p^{2rt\deg_F(u_1) - e(T_{\mathcal{J}})} \\
     &\leq \sum_{\substack{\mathcal{J}\subseteq \mathcal{F}^\pi:\\ \kappa_\mathcal{J}\geq 1}} M'_{\mathcal{J}} p^{2rt\deg_F(u_1) -
       e(T_{\mathcal{J}})} \ll (Mp^{rt\deg_F(u_1)})^2 \stackrel{(\ref{eq:upper-bound-exp-z})}{=} \E[Z]^2.
   \end{split}
 \end{equation*}
 By the second moment method this implies that a.a.s.\ $Z= \Theta\bigl(n^{\kappa(\mathcal{F}^\pi)}p^{e(\mathcal{F}^\pi)}\bigr)$, and that thus at least this number of copies of $\mathcal{F}^\pi$ are created in the second round.
\end{proof}

\subsection{Proofs of the technical lemmas}

The proofs in this section are essentially line-by-line translations of the
analogous proofs in Mütze et al.~\cite{Mutze2011} from the edge to the
vertex case.

Together with Theorem~\ref{thm:achlioptas-ubound} the following lemma proves
Theorem~\ref{thm:threshold-equivalence}.

\begin{lemma}
  \label{lem:threshold-equivalence}
  Let \(F\) be a fixed nonempty graph, and let \(r\geq 2\) be a fixed
  integer. Let \(\theta^* = \theta^*(F,r)\) be the unique solution of
  \begin{equation}
    \Lr(F) \stackrel{!}{=} 0,
  \end{equation}
  where \(\Lr(F)\) is defined in \eqref{eq:lambda-r} and
  \eqref{eq:Lambda-r-theta}. Then we have
  \begin{equation*}
    m^{r*}(F) = \frac{1}{\theta^*(F,r)}.
  \end{equation*}
\end{lemma}

\begin{proof}
  For any nonempty ordered graph \((F,\pi)\), set
  \begin{multline}\label{eq:def-cvH}
      \cvH(F,\pi) := \bigl\{\vH = \bigl((H_1,\sigma), H_2, \dotsc, H_h) \mid
      H_1 \subseteq F \\\wedge \sigma = \pi|_{H_1} = (u_1, \dotsc, u_h) \\\wedge
      \forall i \geq 2 : (H_i \subseteq H_1\setminus\{u_1, \dotsc, u_{i-1}\}
      \wedge u_i \in H_i)\bigr\}
  \end{multline}
  (cf.\ the maximizations in \eqref{eq:def-drstar} and
  \eqref{eq:def-mrstar}). For all \(\vH \in \cvH(F,\pi)\) we define
  \begin{equation}
    \label{eq:e-v-rstar}
    \begin{split}
      e^{r*}(\vH) &:= \sum_{i = 1}^{h} c_i e(H_i), \\
      v^{r*}(\vH) &:= 1 + \sum_{i = 1}^{h} c_i \bigl(v(H_i)-1\bigr),
    \end{split}
  \end{equation}
  where the coefficients \(c_i = c_i(\vH, r)\) are defined as in
  \eqref{eq:def-ci}. Furthermore, we define
  \begin{equation}
    \label{eq:mu-star}
    \mu_{r,\theta}^*(\vH) := v^{r*}(\vH) - \theta\cdot e^{r*}(\vH).
  \end{equation}
  Note that by the definitions in~\eqref{eq:def-mrstar} and \eqref{eq:def-drstar}, we have
  \begin{equation*}
	\begin{split}
    m^{r*}(F,\pi)&=\min_{\pi\in\Pi(V(F))}\max_{H_1\subseteq F} \drs(H_1, \pi|_{H_1})\\
   &=  \min_{\pi\in\Pi(V(F))}\max_{\vH \in \cvH(F,\pi)}
    \frac{e^{r*}(\vH)}{v^{r*}(\vH)} = \frac{1}{\theta^{**}(F,r)},
		\end{split}
  \end{equation*}
  where \(\theta^{**}(F,r)\) is the unique solution of
  \begin{equation*}
     \max_{\pi\in\Pi(V(F))}\min_{\vH \in \cvH(F,\pi)} \mu_{r,\theta}^*(\vH) \stackrel{!}{=} 0.
  \end{equation*}
  To prove Lemma~\ref{lem:threshold-equivalence}, it suffices to show that the left hand side of the last equation equals $\Lambda_{r,\theta}(F)$ as defined in \eqref{eq:Lambda-r-theta}. 
  We will do so by showing that for any nonempty ordered graph \((F,\pi)\) and any
  \(r\geq 2\) and \(0 \leq \theta \leq 2\) we have
  \begin{equation}\label{eq:min-mu=min-lambda}
    \min_{\vH \in \cvH(F,\pi)} \mu_{r,\theta}^*(\vH) = \min_{H\subseteq F} \lr(H,\pi|_H).
  \end{equation}
  The remainder of the proof is devoted to establishing \eqref{eq:min-mu=min-lambda}. To simplify the notation we consider \(r\) and \(\theta\) fixed and drop all
  corresponding sub- and superscripts. In the following equations we define
  the quantities \(\tilde e\), \(\tilde v\) and \(\tilde \mu\), which depend
  on the choice of an ordered graph \((H_1,\sigma)\). In principle we should
  write \(\tilde e_{(H_1,\sigma)}\), \(\tilde v_{(H_1, \sigma)}\) and \(\tilde
  \mu_{(H_1, \sigma)}\), but we omit this dependency from the notation as well.

  Consider the following recursive definitions for \(1 \leq i \leq h\):
  \begin{equation}
    \label{eq:e-v-tilde}
    \begin{split}
      \tilde{e}(H_i, \dotsc, H_h) &:= e(H_i) + (r-1)\cdot\sum_{j=i+1}^{h}\ind{u_j \in
        H_i}\tilde e(H_j, \dotsc, H_h)\\
      \tilde{v}(H_i, \dotsc, H_h) &:= v(H_i) -1 + (r-1)\cdot\sum_{j=i+1}^{h}\ind{u_j \in
        H_i}\tilde v(H_j, \dotsc, H_h).
    \end{split}
  \end{equation}
  We can now write \(e^{*}(\vH)\) and \(v^{*}(\vH)\) as
  \begin{equation}
    \label{eq:evstar-by-tilde}
    \begin{split}
      e^{*}(\vH) &= r\cdot\tilde e(H_1, \dotsc, H_h) \\
      v^{*}(\vH) &= 1 + r\cdot\tilde v(H_1, \dotsc, H_h).
    \end{split}
  \end{equation}
  This can be verified by induction, using the definition of $c_i$ in~\eqref{eq:def-ci} and noting that for \(1\leq k \leq h\) we
  have
  \begin{equation*}
    \begin{split}
      e^*(\vH) &= \sum_{i = 1}^{k}c_i e(H_i) + (r-1)\cdot
      \sum_{j=k+1}^{h}\Bigl(\sum_{i=1}^{k}c_i \ind{u_j \in H_i}\Bigr)\tilde
      e(H_j, \dotsc, H_h)\\
      v^*(\vH) &= 1 + \sum_{i = 1}^{k}c_i \bigl(v(H_i)-1\bigr) + (r-1)\cdot
      \sum_{j=k+1}^{h}\Bigl(\sum_{i=1}^{k}c_i \ind{u_j \in H_i}\Bigr)\tilde
      v(H_j, \dotsc, H_h),
    \end{split}
  \end{equation*}
  which is equivalent to \eqref{eq:e-v-rstar} for \(k = h\) and to
  \eqref{eq:evstar-by-tilde} for \(k = 1\). Combining \eqref{eq:e-v-tilde} and
  \eqref{eq:evstar-by-tilde} via \eqref{eq:mu-star} also yields that
  \begin{equation}\label{eq:mu-star-mu-tilde-relation}
    \mu^*(\vH) = 1 + r\tilde \mu(H_1, \dotsc, H_h),
  \end{equation}
  where
  \begin{equation}\label{eq:tilde-mu-def}
    \tilde\mu(H_i, \dotsc, H_h) := (v(H_i)-1) - \theta e(H_i) +
    (r-1)\sum_{j=i+1}^{h}\ind{u_j \in H_i} \tilde\mu(H_j, \dotsc, H_h).
  \end{equation}

  It follows that for any fixed subgraph \(H_1\subseteq F\) and \(\sigma :=
  \pi|_{H_1} = (u_1, \dotsc, u_h)\) the following holds: for \(1 \leq i \leq
  h\) and any graph \(H_i\subseteq H_1\setminus\{u_1, \dotsc, u_{i-1}\}\) with
  \(u_i\in H_i\), the value
  \begin{equation}\label{eq:lambda-tilde}
    \tilde\lambda_{(H_1,\sigma)}(H_i, i) := \min_{\substack{H_{i+1}, \dotsc, H_h\\\forall j\geq i+1: H_j \subseteq H_1\setminus
      \{u_1, \dotsc, u_{j-1}\}\wedge u_j \in H_j}} \tilde \mu (H_i, \dotsc, H_h)
  \end{equation}
  can be computed recursively via
  \begin{multline}
    \label{eq:lambda-tilde-recursive}
    \tilde\lambda_{(H_1,\sigma)}(H_i, i) = (v(H_i)-1) - \theta e(H_i) \\+
    (r-1)\cdot\sum_{j=i+1}^{h}\ind{u_j\in H_i}\cdot\min_{H_j \subseteq
      H_1\setminus \{u_1, \dotsc, u_{j-1}\}:\, u_j \in H_j} \tilde
    \lambda_{(H_1,\sigma)} (H_j, j).
  \end{multline}
  In the remainder of the proof we simplify the recursion on the right side to
  relate it to \(\lambda()\) as defined in \eqref{eq:lambda-r}. First we show
  that we can get rid of the dependency on \((H_1, \sigma)\), and that the
  value of \(\tilde\lambda_{(H_1,\sigma)}(H_i, i)\) in fact only depends on
  the isomorphism class of \((H_i, \sigma|_{H_i})\). To this end, we prove
  that for any fixed ordered graph \((H_1, \sigma)\) there exists a sequence
  \(H_2, \dotsc, H_h \subseteq H_1\) as in \eqref{eq:def-cvH} minimizing
  \(\tilde\mu(H_1, \dotsc, H_h)\) with the additional property that
  \begin{equation}
    \label{eq:Hi-Hj-subgraph}
    u_{j} \in H_i \Rightarrow H_j \subseteq H_i.
  \end{equation}
  Let \(H_2, \dotsc, H_h \subseteq H_1\) be graphs minimizing \(\tilde\mu(H_1,
  \dotsc, H_h)\) such that every \(H_i\) is inclusion-maximal, and assume for
  the sake of contradiction that there exist indices \(2 \leq i < j\) with
  \(u_j \in H_i\) but \(H_j \nsubseteq H_i\). Our choice of \(H_2, \dotsc,
  H_h\) implies that for \(H_i' := H_i \cup H_j\) and \(H_j' := H_i \cap H_j\)
  we have
  \begin{equation*}
    \begin{split}
      \tilde\mu(H_i', \dotsc, H_h) - \tilde\mu(H_i, \dotsc, H_h) &> 0,\\
      \tilde\mu(H_j, \dotsc, H_h) - \tilde\mu(H_j', \dotsc, H_h) &\leq 0,
    \end{split}
  \end{equation*}
  where the first inequality is strict due to the inclusion-maximality of
  \(H_i\). Expanding the above equations according to~\eqref{eq:tilde-mu-def}
  yields that both terms are equal to
  \begin{equation*}
    \bigl(v(H_j) - v(H_j')\bigr) - \theta\bigl(e(H_j)-e(H_j')\bigr) + (r-1) \sum_{k=j+1}^{h}\ind{u_k \in H_j\setminus H_i}\tilde\mu(H_k,\dotsc,H_h),
  \end{equation*}
  which is a contradiction. W.l.o.g.\ we may therefore assume
  that \eqref{eq:Hi-Hj-subgraph} holds, and that in
  \eqref{eq:lambda-tilde-recursive} we can minimize over
  subgraphs of \(H_i\setminus\{u_i, \dotsc, u_{j-1}\}\) instead of subgraphs
  of \(H_1 \setminus\{u_1, \dotsc, u_{j-1}\}\).

  Observe that in \eqref{eq:lambda-tilde-recursive} the context
  \((H_1,\sigma)\) is now irrelevant, and that we only require the
  ordering \(\sigma|_{H_i}\) on the right hand side. Setting
  \begin{equation}\label{eq:tilde-lambda-drop-context}
    \tilde\lambda_{(H_1,\sigma)}(H_i,i) =: \tilde\lambda(H_i,\sigma|_{H_i}).
  \end{equation}
  and changing notations accordingly, we obtain from \eqref{eq:lambda-tilde-recursive}
  \begin{multline}\label{eq:lambda-tilde-context-free-sum}
    \tilde\lambda(H, \tau =: (u_1, \dotsc, u_h)) = \bigl(v(H)-1\bigr) - \theta
    e(H) \\+ (r-1)\cdot \sum_{j=2}^{h} \min_{J\subseteq H\setminus\{u_1,\dotsc,
      u_{j-1}\}: u_j \in J}\tilde\lambda(J, \tau|_{J}).
  \end{multline}
  Next we get rid of the sum in the equation above as follows:
  \begin{equation}\label{eq:lambda-tilde-context-free}
    \begin{split}
      \tilde\lambda(H, \tau) &= \bigl(v(H)-1\bigr) - \theta e(H) \\
      &\quad+ (r-1)\min_{J\subseteq H\setminus u_1 : u_2 \in J} \tilde\lambda(J,\tau|_{J})\\
      &\quad+ (r-1)\cdot \sum_{j=3}^{h} \min_{J\subseteq
        H\setminus\{u_1,\dotsc, u_{j-1}\}: u_j \in J}\tilde\lambda(J,
      \tau|_{J}).\\
      &=(v(H\setminus u_1)-1) + 1 - \theta e(H\setminus u_1) - \theta\deg_{H}(u_1)\\
      &\quad+ (r-1)\min_{J\subseteq H\setminus u_1 : u_2 \in J} \tilde\lambda(J,\tau|_{J})\\
      &\quad+ (r-1)\cdot \sum_{j=3}^{h} \min_{J\subseteq
        H\setminus\{u_1,\dotsc, u_{j-1}\}: u_j \in J}\tilde\lambda(J,
      \tau|_{J}).\\
      &= 1 + \tilde\lambda(H\setminus u_1, \tau|_{H\setminus u_1}) - \theta
      \deg_H(u_1) + (r-1)\min_{\substack{J\subseteq H\setminus u_1:\\u_2\in J}}\tilde\lambda(J,\tau|_{J}).
    \end{split}
  \end{equation}
  Substituting
  \begin{equation}\label{eq:tilde-bar-substitutions}
    \tilde\lambda(H,\tau) =: \bar\lambda(H\setminus u_1,
    \tau \setminus u_1)-\theta\deg_{H}(u_1), \quad
  (H\setminus u_1, \tau\setminus u_1) =: (\bar H, \bar\tau), \quad
  u_2 =: \bar u_1
  \end{equation}
  we see that the last line of \eqref{eq:lambda-tilde-context-free} is
  equivalent to
  \begin{multline*}
    \bar\lambda(\bar H,\bar \tau) = 1 + \bar\lambda(\bar H\setminus \bar u_1,
    \bar\tau \setminus \bar u_1) - \theta\deg_{\bar H}(\bar u_1)\\ +
    (r-1)\min_{\substack{J\subseteq \bar H:\\\bar u_1 \in J}}
    \Big(\bar\lambda(J\setminus \bar u_1, \bar\tau|_{J\setminus \bar u_1}) -
    \theta\deg_{J}(\bar u_1)\Big),
  \end{multline*}
  which is the recursive step in the definition of \(\lambda()\) in
  \eqref{eq:lambda-r}. Moreover if \((\bar H,\bar \tau) = (H\setminus u_1,
  \tau\setminus u_1)\) contains no vertices (i.e.\ \(H\) is a graph on 1
  vertex and therefore no edges) we have
  \begin{equation*}
    \bar\lambda(\bar H,\bar \tau) \stackrel{\eqref{eq:tilde-bar-substitutions}}{=} \tilde\lambda(H, \tau) + \theta\deg_H(u_1)
    = \tilde\lambda(H,\tau)
    \stackrel{\eqref{eq:lambda-tilde-context-free-sum}}{=} 0 = \lambda(\bar H,
    \bar \tau).
  \end{equation*}
  This takes care of the base case and implies that \(\bar\lambda(H,\tau) =
  \lambda(H,\tau)\) for all ordered graphs \((H,\tau)\). Thus we have for
  every fixed \((H_1, \sigma)\), \(\sigma = (u_1, \dotsc, u_h)\), that
  \begin{multline}\label{eq:tilde-mu-is-tilde-lambda-is-lambda}
      \min_{\substack{H_2, \dotsc, H_h\\\forall j\geq 2: H_j \subseteq
          H_1\setminus \{u_1, \dotsc, u_{j-1}\}\wedge u_j \in H_j}} \tilde \mu
      (H_1, \dotsc, H_h) \stackrel{\eqref{eq:lambda-tilde}}{=} \tilde\lambda_{(H_1,\sigma)}(H_1, 1) \\\stackrel{\eqref{eq:tilde-lambda-drop-context}}{=}
      \tilde\lambda(H_1, \sigma) \stackrel{\eqref{eq:tilde-bar-substitutions}}{=} \bar\lambda(H_1\setminus u_1,
      \sigma\setminus u_1) - \theta\deg_{H_1}(u_1) \\= \lambda(H_1\setminus u_1,
      \sigma\setminus u_1) - \theta\deg_{H_1}(u_1).
  \end{multline}
  Still using the notation \(\pi|_{H_1} = \sigma = (u_1, \dotsc,  u_h)\)
  (cf.~\eqref{eq:def-cvH}), equation \eqref{eq:min-mu=min-lambda} now follows
  from
  \begin{multline*}
    \min_{\vH \in \cvH(F,\pi)} \mu_{r,\theta}^*(\vH)\\
    \stackrel{\eqref{eq:def-cvH},\eqref{eq:mu-star-mu-tilde-relation}}{=}
    \min_{H_1\subseteq F}\Bigl\{ 1 + r\cdot \min_{\substack{H_{2}, \dotsc,
        H_h\\\forall j\geq 2: H_j \subseteq H_1\setminus
        \{u_1, \dotsc, u_{j-1}\}\wedge u_j \in H_j}} \tilde\mu(H_1, \dotsc, H_h)\Bigr\}\\
    \stackrel{\eqref{eq:tilde-mu-is-tilde-lambda-is-lambda}}{=}
    \min_{H_1\subseteq F}\Bigl\{1 + r\bigl(\lambda(H_1\setminus u_1, \sigma\setminus
    u_1) - \theta\deg_{H_1}(u_1)\bigr)\Bigr\}\\
    \stackrel{\eqref{eq:min-nice-recursion}}{=} \min_{H_1\subseteq F}
    \lr(H_1,\sigma) = \min_{H\subseteq F} \lr(H,\pi|_H),
  \end{multline*}
  where in the last step we applied Lemma~\ref{lem:lr-drop-minimization} to
  the family of all ordered subgraphs of \((F,\pi)\).
\end{proof}

It remains to prove Lemma~\ref{lem:subsets-of-Fpi-are-not-dense}.

\begin{proof}[Proof of Lemma~\ref{lem:subsets-of-Fpi-are-not-dense}]
  \newcommand{\Fpirec}[1]{(F_{#1-})^{\pi_{#1-}}}
  \newcommand{\hFpirec}[1]{(\hat{F}_{#1-})^{\pi_{#1-}}}
  For this proof we require an extension of the definition of connectedness to
  \(r\)-matched graphs. We call an \(r\)-matched graph \(H = (V, E,
  \mathcal{K})\) \emph{connected} if for any 2 vertices \(u, v \in V\) which
  are not part of the same \(r\)-set, there exists a sequence of \(r\)-sets
  \(K_1, \dotsc, K_t \in \mathcal{K}\) such that \(u \in K_1\), \(v \in K_t\)
  and there exists an edge between at least one vertex in \(K_i\) and one in
  \(K_{i+1}\) for all \(1 \leq i \leq t-1\). Since the value of
  \(\mu_{r,\theta}()\) for a disconnected \(r\)-matched graph \(H\) is simply
  the sum of the values of \(\mu_{r,\theta}()\) for all connected components
  of \(H\), it suffices to prove the claim for all connected \(r\)-matched
  subgraphs \(J \subseteq F_r^\pi\), i.e.\ to prove that for any integer \(r \geq
  2\) and any \(0 \leq \theta \leq 2\) we have
  \begin{equation}
    \label{eq:min-connected}
    \min_{\substack{J \subseteq \Fpi:\\\text{\(J\) connected}}} \mr(J) =
    \min_{\substack{H\subseteq F}} \lambda_{r,\theta}(H, \pi|_{H}).
  \end{equation}
  For the remainder of the proof we consider \(r\) and \(\theta\) fixed and
  drop the corresponding subscripts from the notation.

  Let \(\pi = (u_1, \dotsc, u_f)\), and define \(F_{i-} := F \setminus \{u_1,
  \dotsc, u_i\}\) and \(\pi_{i-} := \pi|_{F_{i-}}\). For any grey-black
  \(r\)-matched graph \(\Fpirec{i}\) we call the \(r\)-set containing the
  vertex \(u_{i+1}\) of its central copy of \(F_{i-}\) the \emph{central
    \(r\)-set}.

  Let \(J\) be a connected subgraph of \(F^{\pi}\), and let \(0 \leq i \leq
  f-1\) be the largest index such that \(J\) is also contained in a copy of
  \(\Fpirec{i}\). By the maximal choice of \(i\) and the
  connectedness of \(J\), the graph \(J\) contains the central \(r\)-set of
  this copy. With this we can reformulate equation \eqref{eq:min-connected} to
  \begin{equation}
    \label{eq:min-mu-equals-min-lambda}
    \min_{0\leq i\leq f-1} \min_{\substack{J \subseteq
        \Fpirec{i}:\\\text{\(K(u_{i+1})\in J \wedge J\)
          connected}}} \mu(J) =
    \min_{0\leq i\leq f-1} \min_{\substack{H\subseteq F_{i-}:\\u_{i+1}\in H}} \lambda(H, \pi|_{H}).
  \end{equation}
  where we use \(K(u_{i+1}) \in J\) as a shorthand notation to indicate that
  \(J\) contains the central \(r\)-set \(K(u_{i+1})\) of
  \(\Fpirec{i}\). We now show that the inner minimizations of
  \eqref{eq:min-mu-equals-min-lambda} are equivalent. By changing variables
  (\(F \leftarrow F_{i-}\) and \(\pi \leftarrow \pi_i\)) this reduces to
  showing that for any ordered graph \((F,\pi)\) we have
  \begin{equation}\label{eq:lem-J-subgraphs-reduced}
    \min_{\substack{J \subseteq F^\pi : K(u_1) \in J\\ \text{\(J\) connected}}} \mu(J) =
    \min_{\substack{H\subseteq F :\\u_{1}\in H}} \lambda(H, \pi|_{H}).
  \end{equation}

  For any \(r\)-matched graph \(H\) we refer to a subgraph \(J \subseteq H\)
  that minimizes \(\mu(J)\) as a \emph{rarest subgraph} of \(H\). To determine
  a rarest subgraph of \(F^\pi\) we can make use of its recursive structure.

  Let $1\leq i\leq f-1$, and consider a fixed copy \(\hFpirec{i}\) of \(\Fpirec{i}\) in
  \(F^\pi\). By \(\hat F_{i-}\) we denote the central copy of \((F_{i-},
  \pi_{i-})\) in \(\hFpirec{i}\), and by \(\hat u_i\) the vertex that completes
  \(\hat F_{i-}\) to a copy of \((F_{(i-1)-}, \pi_{(i-1)-})\). Moreover, let \(\hFpirec{(i-1)}\) denote   the copy of \(\Fpirec{(i-1)}\) that is formed by  \(\hFpirec{i}\), \(K(\hat u_i)\) and $r-1$ other copies of \(\Fpirec{i}\).

  Note that the \(r\) copies of \(\Fpirec{i}\)
  joined at the central \(r\)-set \(K(\hat u_i) = (u_{i,1},\dotsc, u_{i,r})\)
  of \(\hFpirec{(i-1)}\) are essentially independent: For each \(u_{i,\ell}\),
  \(1 \leq \ell \leq r\), we consider the graph obtained by removing from
  \(\hFpirec{(i-1)}\) the $r-1$ copies of \(\Fpirec{i}\) that are \emph{not} associated with \(u_{i,\ell}\), i.e., whose central copy of
  \((F_{i-}, \pi_{i-})\) is \emph{not} connected by $\deg_{F_{(i-1)-}}(u_i)$ edges to $u_{i,\ell}$. (If $\deg_{F_{(i-1)-}}(u_i)=0$ we associate the copies of \(\Fpirec{i}\) with the vertices of $K(\hat u_i)$ arbitrarily.) We call this graph the \emph{branch} of \(\hFpirec{(i-1)}\) corresponding to \(u_{i,\ell}\).
  Note that this is still an \(r\)-matched graph and that all
  \(r\) branches contain the central \(r\)-set \(K(\hat u_i)\). By
  the linearity of \(\mu(H)\) in \(e(H)\) and \(\kappa(H)\), a rarest
  connected subgraph of \(\hFpirec{(i-1)}\) containing \(K(\hat u_i)\) can
  be found by determining a rarest connected subgraph containing \(K(\hat
  u_i)\) in each branch of \(\hFpirec{(i-1)}\) independently. Let \(\hat
  J_i\) denote an arbitrary fixed such rarest subgraph.
  Note that in particular we can compute \(\mu(J)\) as on the left hand side of \eqref{eq:lem-J-subgraphs-reduced} as
  \begin{equation}\label{eq:mu-J-J-1-recursion}
    \mu(J) = 1 + r(\mu(\hat J_1) - 1).
  \end{equation}
  Similarly, we can find a rarest connected subgraph \(J_i\) containing \(K(\hat u_i)\)
  for a branch of \(\hFpirec{(i-1)}\) by determining an optimal choice for
  \(H_i := J_i \cap \hat F_{(i-1)-}\). For any choice of $H_i$, by recursion, for each vertex $u'_j$ of $H_i$, \(i+1 \leq j \leq f\), we already
  know a rarest subgraph containing \(K(u_j')\) for all the \(r-1\)
  remaining branches of the copy of \(\Fpirec{(j-1)}\) corresponding to the
    other \(r-1\) vertices of \(K(u_j')\). Letting \(\hat J_j\), \(i+1\leq j
    \leq f\) denote such rarest subgraphs, we obtain that the value of $\mu(J_i)$ resulting from a given choice of \(H_i\subseteq\hat F_{(i-1)-}\) is
  \begin{equation}\label{eq:mu-J-recursion}
    \mu(J_i) = v(H_i) - \theta e(H_i) + \sum_{j = i+1}^{f}\ind{u_i \in H_j} (r-1)(\mu(\hat J_j)-1).
  \end{equation}
Here we used that for \(i+1 \leq j \leq
    f\) all \(r-1\) many copies of \(\hat J_j\) share one \(r\)-set, and that each such
    \(r\)-set also contains one vertex of \(v(H_i)\).

  Substituting \(\mu(J_i)-1 =: \tilde\lambda_{(F,\pi)}(H_i,i)\) in the above
  equation yields for \(1 \leq i \leq f\) the recursion
  \begin{multline*}
    \tilde\lambda_{(F,\pi)}(H_i,i) = (v(H_i)-1) - \theta e(H_i) \\+ (r-1)\sum_{j =
      i+1}^{f}\ind{u_j \in H_i}\min_{\substack{H_j \subseteq
        H_1\setminus\{u_1, \dotsc, u_{j-1}\}:\\u_j \in H_j}}
    \tilde\lambda_{(F,\pi)}(H_j,j).
  \end{multline*}
  This is essentially the same recursion as \eqref{eq:lambda-tilde-recursive}
  in the proof of Lemma~\ref{lem:threshold-equivalence}.
  Analogously to the proof of Lemma~\ref{lem:threshold-equivalence} one can show that
  \begin{equation}\label{eq:tilde-lambda-to-lambda}
    \tilde\lambda_{(F,\pi)}(H_1, 1) = \lambda(H_1\setminus u_1, \sigma\setminus u_1) - \theta\deg_{H_1}(u_1),
  \end{equation}
  where \(\sigma := \pi|_{H_1}\)
  (cf.~\eqref{eq:tilde-mu-is-tilde-lambda-is-lambda}). Finally
  \begin{equation*}
    \begin{split}
      \min_{\substack{J\subseteq F^\pi: K(u_1) \in J\\\wedge J \text{
            connected}}} \mu(J) &\stackrel{\eqref{eq:mu-J-J-1-recursion}}{=}
      r\bigl(\mu(\hat J_1)-1\bigr) + 1 = \min_{H_1\subseteq F: u_1 \in
        H} 1+r\cdot \tilde\lambda(H_1, 1)
      \\&\stackrel{\eqref{eq:tilde-lambda-to-lambda}}{=} \min_{H_1\subseteq F:
        u_1 \in H_1} 1+r\cdot \bigl(\lambda(H_1\setminus u_1, \sigma\setminus
      u_1) - \theta\deg_{H_1}(u_1)\bigr)
      \\&\stackrel{~\eqref{eq:min-nice-recursion}}{=} \min_{H_1\subseteq F:
        u_1 \in H_1} \lambda(H_1, \sigma) = \min_{H\subseteq F: u_1 \in H}
      \lambda(H, \pi|_{H}).
    \end{split}
  \end{equation*}
  In the last line we applied Lemma~\ref{lem:lr-drop-minimization} to the
  family of all ordered subgraphs of \((F, \pi)\) that contain the vertex
  \(u_1\). This shows \eqref{eq:lem-J-subgraphs-reduced} and finishes the
  proof.
\end{proof}

\end{document}